\documentclass[12pt,oneside]{amsart}
\usepackage{graphics}
\usepackage[dvips]{graphicx}
\usepackage{cancel}
\usepackage{amsmath,amsfonts}
\usepackage[update,prepend]{epstopdf}
\usepackage{tikz}
\usepackage{times}
\usepackage{float}
\usepackage[ansinew]{inputenc}
\usepackage{amssymb}
\usepackage{amsmath}
\usepackage{subfigure}
\usepackage{amsfonts}
\usepackage{latexsym}
\usepackage{epsfig}
\usepackage{geometry}
\usepackage[update,prepend]{epstopdf}
\usepackage[colorlinks = true,
linkcolor = black, citecolor=black]{hyperref}
\def\2010mathclass#1{\par%
	\insert\footins{\footnotesize{\it\textup{2010} Mathematics
			Subject Classification:} #1}}
		
\numberwithin{equation}{section}
\newtheorem{Theorem}{Theorem}[section]
\newtheorem{Definition}{Definition}[section]

\newtheorem{Proposition}{Proposition}[section]
\newtheorem{Lemma}{Lemma}[section]

\newtheorem{assum}{Assumption} [section]

\newtheorem{Remark}{Remark}[section]

\begin{document}
\title[Kelvin-Voigt]{Exponential decay for the semilinear wave equation with localized Kelvin-Voigt damping}

\author[Astudillo Rojas]{María Rosario Astudillo Rojas}
\address{Institute of Science and Technology, Federal University of São Paulo, 12247-014, São José dos Campos, SP, Brazil}
\email{mastudillo86@gmail.com}

\author[Cavalcanti]{Marcelo M. Cavalcanti}
\address{ Department of Mathematics, State University of
Maring\'a, 87020-900, Maring\'a, PR, Brazil.}
\email{mmcavalcanti@uem.br}
\thanks{Research of Marcelo M. Cavalcanti partially supported by the CNPq Grant
300631/2003-0}

\author[Corr\^{e}a]{Wellington J. Corr\^ea}
\address{ Academic Department of Mathematics, Federal Technological University of  Paran\'a, Campuses Campo Mour\~{a}o,  87301-899, Campo Mour\~{a}o, PR, Brazil.}
\email{wcorrea@utfpr.edu.br}
\thanks{Research of Wellington J. Corrêa partially supported by the CNPq Grant 438807/2018-9}

\author[Domingos Cavalcanti]{Val\'eria N. Domingos Cavalcanti}
\address{ Department of Mathematics, State University of Maring\'a,
87020-900, Maring\'a, PR, Brazil.}
\email{vndcavalcanti@uem.br}
\thanks{Research of Val\'eria N. Domingos Cavalcanti partially supported by the CNPq Grant
304895/2003-2}

\author[Gonzalez Martinez]{Victor H. Gonzalez Martinez}
\address{Department of Mathematics, State University of
	Maring\'a, 87020-900, Maring\'a, PR, Brazil.}
\email{victor.hugo.gonzalez.martinez@gmail.com}
\thanks{Research of Victor Hugo Gonzalez Martinez supported by CAPES}

\author[Vicente Andr\'e]{  Andr\'{e} Vicente}
\address{Centro de Ci\^{e}ncias Exatas e Tecnol\'{o}gicas - Universidade Estadual do Oeste do Paran\'{a}, Cascavel-PR, Brazil.}
\email{andre.vicente@unioeste.br}


\2010mathclass{35L05, 35l20, 35B40.}%

\maketitle

\begin{abstract}
In the present paper, we are concerned with the semilinear viscoelastic wave equation subject to a locally distributed dissipative effect of Kelvin-Voigt type, posed on a bounded domain with smooth boundary.   We begin with an auxiliary problem and we show that its solution decays exponentially in the weak phase space.  The method of proof combines an observability inequality and unique continuation properties. Then, passing to the limit, we recover the original model and prove its global existence as well as the exponential stability.
\end{abstract}

\maketitle

\tableofcontents

\section{Introduction}

\subsection{Description of the Problem.}

This article is devoted to the analysis of the exponential and uniform decay rates
of solutions to the wave equation subject to a Kelvin-Voigt damping
\begin{equation}\label{eq:*}
\left\{
\begin{aligned}
&{\partial_t^2 u -\Delta u + f(u) - \operatorname{div}(a(x) \nabla \partial_t u) = 0\quad \hbox{in}\,\,\,\Omega \times (0, +\infty),}\\\
&{u=0\quad \hbox{on}\quad \partial \Omega \times (0,+\infty ),}\\\
&{u(x,0)=u_0(x)\in H_0^1(\Omega);\quad \partial_t u(x,0)=u_1(x)\in L^2(\Omega),\quad x\in\Omega,}
\end{aligned}
\right.
\end{equation}
where $\Omega$ is a bounded domain of $\mathbb{R}^n ,$ $n\geq 1$, with smooth boundary $\partial \Omega$,~ $f:\mathbb{R}\rightarrow \mathbb{R}$ is a $C^2$ function with sub-critical growth which satisfies the sign condition $f(s) s \geq 0$, for all $s\in \mathbb{R}$ (see further assumptions \eqref{ass on f} and \eqref{ass on F}).

The following assumptions are made on the function $a(x)$, responsible for the localized dissipative effect of Kelvin-Voigt type:
\begin{assum}
We assume that $a(\cdot) \in L^\infty(\Omega)$ is a nonnegative function. In addition, there exists a compact, connected set $A \subset \Omega$ with smooth boundary and non-empty interior, verifying
$$A:=\{x\in \Omega: a(x)=0\},$$
We also assume that $a\in C^0(\overline{\omega})$, where $\omega:= \Omega \backslash A $.
\end{assum}

\begin{Remark}\label{Remark1}
All the results in this paper are also true for non constant coefficients (with the same proof). We could e.g. replace $\Delta$ by $\Delta_G=\frac{1}{\rho(x)}\operatorname{div}(K(x)\nabla u)$ with $\rho(x)=\sqrt{\operatorname{det}(g_{ij})}$, where $G=(K/\rho)^{-1}$ and $K(x)$ is a symmetric positive-definite matrix such that $\alpha|\xi|^2 \leq \xi^{\top}\cdot K(x) \cdot \xi \leq \beta |\xi|^2,$ for all $\xi \in \mathbb{R}^{d}$ and $\alpha,\beta$ are positive constants, see \cite{ademirmaria}. In this case, all the geodesics of the metric $G$ will enter in the set $\omega=\Omega \setminus A$ according to Figure \ref{Figure1}.
\begin{figure}[H]
	\centering
\includegraphics[scale=0.33]{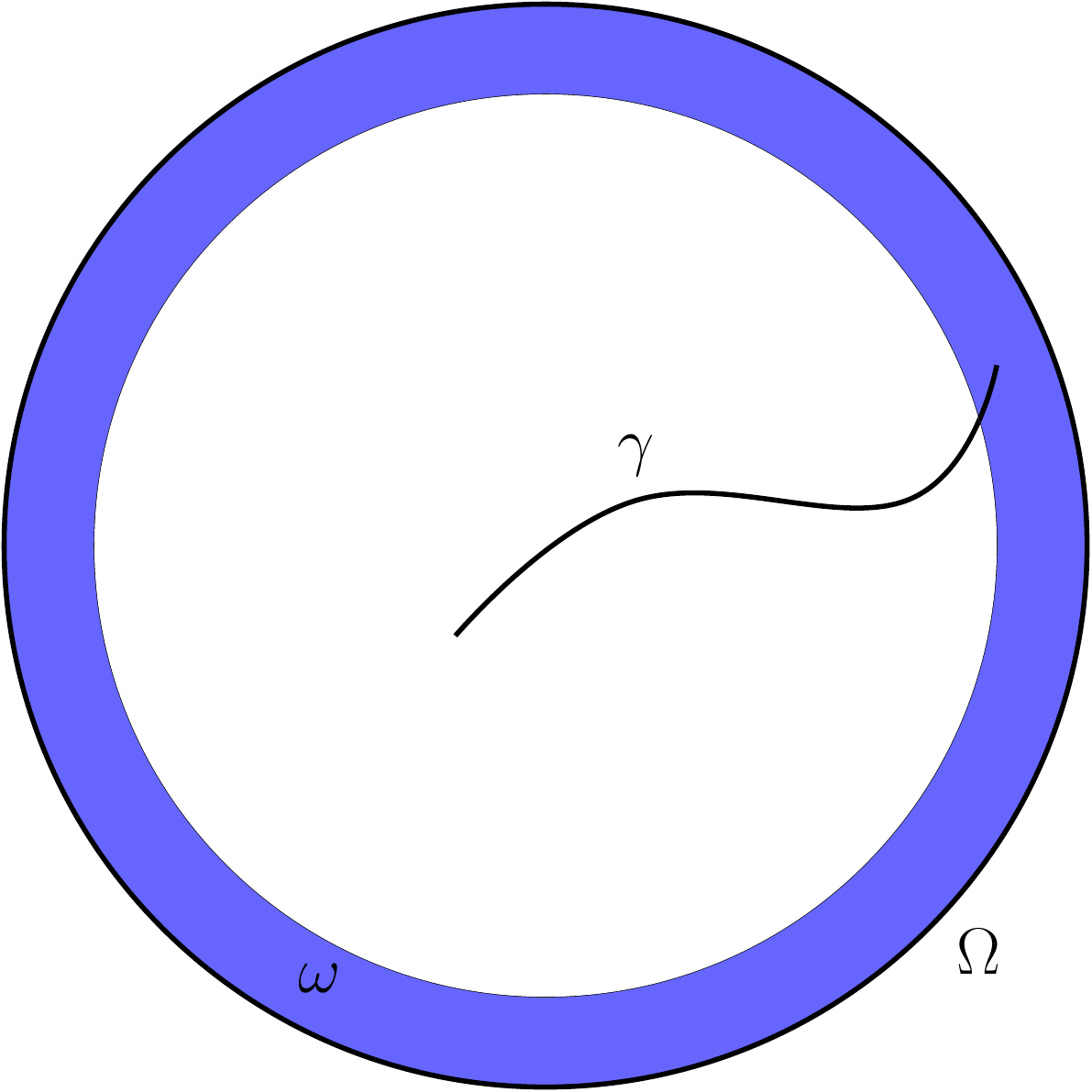}
\caption{\small It is possible to choose the density function $\rho(x) $ so that every geodesic curve $\gamma$ (in black) of the metric $G=(K/\rho)^{-1}$ enters in the damped area $\omega$ (in blue) satisfying the Geometric Control Condition. Thus, exponential decay rate estimates are expected in this case.}
\label{Figure1}
\end{figure}
\end{Remark}

When we do not have any control on the geodesics of the metric $G=(K/\rho)^{-1},$ we have to assume damping everywhere on $\Omega,$ satisfying the following assumptions:
\begin{itemize}
	\item [i)] For all $x\in \partial \Omega$, $a(x) >0$,
	\item [ii)] For all geodesic $t\in I \mapsto x(t)\in \Omega$ of the metric $G=(K/\rho)^{-1}$, with $0\in I$, there exists $t\geq 0$ such that $a(x(t))>0$.
\end{itemize}
The best way to do this is by using the ideas introduced in Cavalcanti et al. \cite{Cavalcanti2, Cavalcanti3}, namely, $a(x)\geq a_0$ in a neighborhood, $\omega,$ of the boundary $\partial \Omega,$ while $a(x) \geq a_0^\ast>0$ in $(\Omega\backslash\omega) \backslash V$, where $V=\bigcup_{i=1}^{k}V_i$ and $\operatorname{meas}(V) \geq \operatorname{meas}(\Omega\backslash\omega)-\varepsilon,$ for an arbitrary $\varepsilon>0$, according to Figure \ref{Figure3}.
\begin{figure}[H]
	\centering
	\includegraphics[scale=0.33]{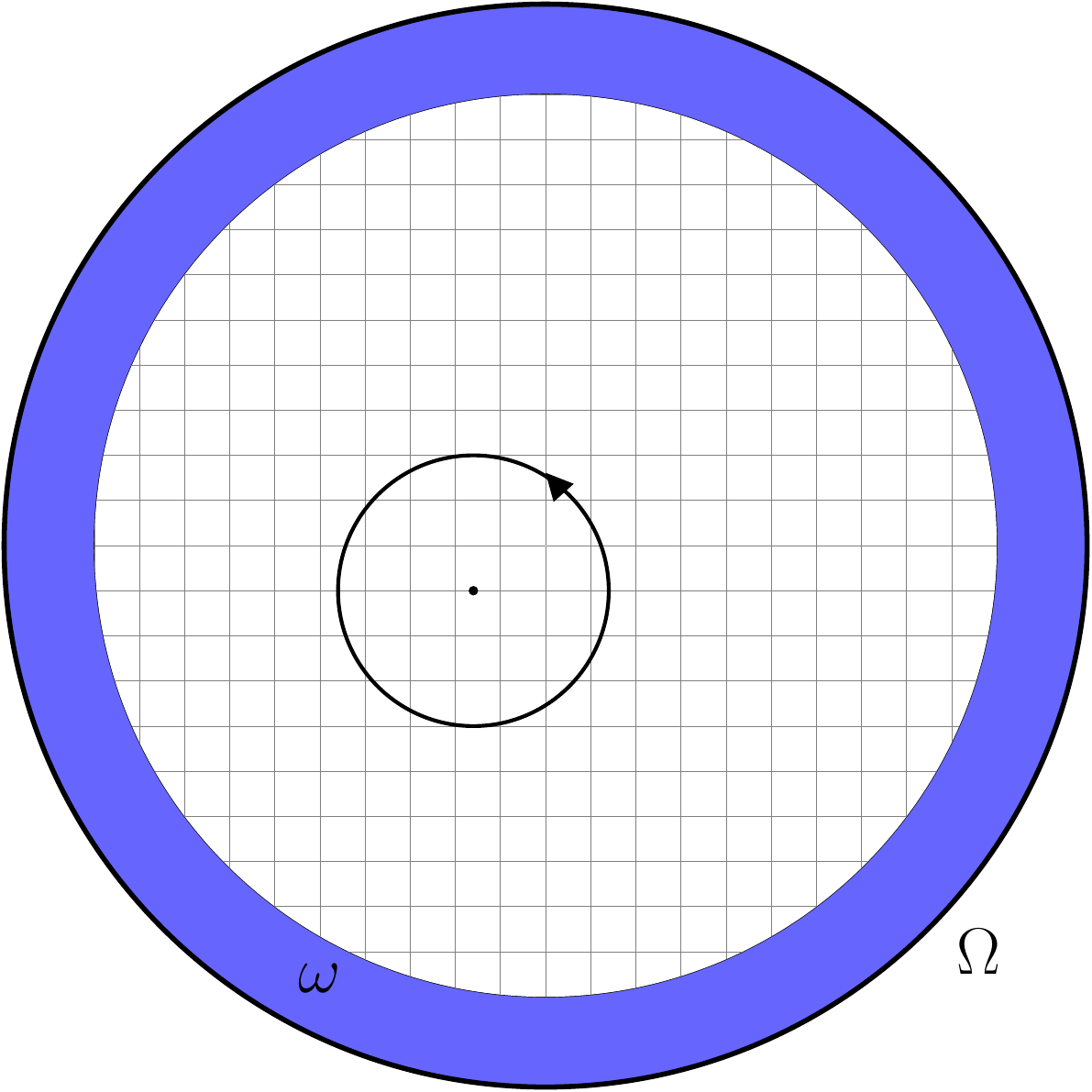}
	\caption{\small The demarcated region $\omega$ (in blue) and $(\Omega\backslash\omega) \backslash V$ (in light gray) illustrates the damped region on the manifold $(\Omega, G)$, which can be considered with measure as small as desired, however totally distributed on $\Omega$. The demarcated region $ V:=\bigcup_{i=1}^{k}V_i$ (in white) illustrates the region without damping with measure arbitrarily large but also totaly distributed in $\Omega$.}
	\label{Figure3}
\end{figure}

The paper is organized as follows.  In section 2 we give the assumptions on the function $f$, we prove some auxiliary results and we pass to the limit in the auxiliary problem.  Besides, we recover the regularity of the solutions, establishing the well-posedness Theorems.  In section 3 we obtain the observability inequality and the energy identity to the auxiliary approximate problem and we establish the exponential stability to problem (\ref{eq:*}).  Finally, the appendix is devoted to prove the convergence of the Kelvin-Voigt term and we also present some results regarding microlocal defect measures which will be useful to prove the stability result.

\subsection{Previous Results, Main Goal and Methodology.}

The decay properties of solutions to the wave equation have been widely studied by many authors under different conditions. Among the numerous papers regarding the wave equation we mention the following references: \cite{alabau:05}, \cite{Bardos}, \cite{Bortot2013}, \cite{Burq-Gerard-CR},  \cite{Cavalcanti2}, \cite{Cavalcanti3}, \cite{Dehman0}, \cite{Dehman}, \cite{Dehman2}, \cite{Gerard-JFA}, \cite{Haraux}, \cite{Joly}, \cite{Laurent1}, \cite{Lebeau}, \cite{martinez:99:RMC}, \cite{Miller}, \cite{Nakao}, \cite{RT}, \cite{Ruiz}, \cite{Tebeou-localized}, \cite{toundykov:07} and \cite{Zuazua} and a long list of references therein.

The study of problem (\ref{eq:*}) presents two main points of difficulty. The first one is to deal with the viscoelastic damping of Kelvin-Voigt type, which generates an unbounded operator.  In addition, the domain consists of two different materials, that is, there is an interaction between the elastic component (the portion of $\Omega$ where $a \equiv 0$) and the other one which is the viscoelastic component (the portion of $\Omega$ where $a>0$).

When $f \equiv 0$, the damped wave equation subject to a locally distributed viscoelastic effect of Kelvin-Voigt type has been studied, in the existing literature, by many authors, for instance, \cite{BC}, \cite{Liu0}, \cite{Liu}, \cite{Tebou} and references therein.  In \cite{Liu0},  Liu and Liu consider the problem posed in an interval $(0,L), 0<L<+ \infty$, and they prove that if  the damping coefficient $a=a(x) \in L^{\infty}(0,L)$ is effective in a subset $(\alpha,\beta)$ , such that$0<\alpha< \beta <L$, the energy does not decay uniformly.

In \cite{Cavalcanti4}, the authors studied an equation of the type $y_{tt}-\Delta y+a(x)y_t-\operatorname{div}(b(x)\nabla y_t)=0$ and by using a combination of the multiplier techniques and the frequency domain method, they show that a convenient interaction of the two damping mechanisms is powerful enough for the exponential stability of the dynamical system, provided that the coefficient of the Kelvin-Voigt damping is smooth enough and satisfies a structural condition.

In \cite{Liu}, Liu and Rao study the problem posed in higher dimensions and, even considering more regular damping coefficient and smoother initial data, they conclude that there is a loss of regularity which makes it difficult to apply the multiplier method.  To bypass these difficulties, the authors were forced to make several technical assumptions on the damping coefficient.  Tebou  in \cite{Tebou}, relaxed the damping coefficients hypothesis as well as the conditions on the feedback control region, but he had to impose a certain constraint on the gradient of the damping coefficient, which will not be required in our current study.

More recently, Astudillo et al. \cite{Astudillo} proved the exponential stability when $f(s) \ne 0$, in an inhomogeneous medium, for a particular class of density functions $\rho(x)$, which helped  to control the bicharacteristic flow by controlling its projection on the spatial domain, for dimensions $n\geq 2$.

The contribution of the present paper is to introduce a new and a more general approach to obtain the exponential stability of problem (\ref{eq:*}), which generalizes the previous results, and, in addition, can be used for other equations as well regardless of the type of dissipation mechanism considered. In particular, the method allows us to consider an inhomogeneous medium subject to a Kelvin-Voigt type damping acting in a neighborhood of the boundary or in a mesh totally distributed in the domain with measure arbitrarily small.

In order to obtain the desired stability result for the wave equation subject to the Kelvin-Voigt damping, we consider an approximate problem and we show that its solution decays exponentially to zero in the weak phase space. The method of proof combines an observability inequality, microlocal analysis tools and unique continuation properties. Then, passing to the limit, we recover the original model and prove its global existence as well as the exponential stability. The advantage of considering the approximate problem lies in the fact that we do not need to assume a unique continuation principle, as the potential function is essentially bounded and there are well known results in the literature regarding this case.

\medskip

In what follows we are going to explain briefly the methodology we are going to use.

\medskip

Denoting $v=u_{t}$ we may rewrite problem \eqref{eq:*} as the following Cauchy problem in $\mathcal{H}=H_0^{1}(\Omega) \times L^2(\Omega)$
\begin{equation}\label{cauchy problem}
\left\{
\begin{aligned}
&\frac{\partial}{\partial t}(u,v) = \mathcal{A}(u,v) + \mathcal{F}(u,v) \\
\\
& (u,v)(0)= (u_{0}, v_{0}),
\end{aligned}
\right.
\end{equation}
where the linear unbounded operator $\mathcal{A}: D(\mathcal{A}) \rightarrow \mathcal{H}$ is given by
\begin{eqnarray}\label{A'}
\mathcal{A}(u,v) = (v, \operatorname{div}(\nabla u + a(x)\nabla v)),
\end{eqnarray}
with domain
\begin{equation}\label{domain}
D(\mathcal{A}) = \{ (u,v) \in \mathcal{H} : v \in H_{0}^{1}(\Omega), \operatorname{div}( \nabla u + a(x)\nabla v) \in L^{2}(\Omega)\}
\end{equation}
and $\mathcal{F}: \mathcal{H} \rightarrow \mathcal{H}$ is the nonlinear operator
\begin{equation}\label{F}
\mathcal{F}(u,v) = (0, -f(u)).
\end{equation}

As in \cite{Liu}, it is possible to show that the operator $\mathcal{A}: D(\mathcal{A}) \subset  \mathcal{H} \rightarrow \mathcal{H}$ defined by \eqref{A'} and \eqref{domain} generates a $C_{0}$-semigroup of contractions $e^{\mathcal{A}t}$ on the energy space $\mathcal{H}$ and $D(\mathcal{A})$ is dense in $\mathcal{H}$. For more details, see \cite{Pazy}.

Given $\{u_0,u_1\} \in H_0^1(\Omega) \times L^2(\Omega)$, consider a sequence $\{u_{0,k},u_{1,k}\}\in D(\mathcal{A})$, satisfying
\begin{eqnarray}\label{conv init data}
\{u_{0,k},u_{1,k}\} \rightarrow \{u_0,u_1\} ~\hbox{ in } H_0^1(\Omega) \times L^2(\Omega).
\end{eqnarray}

 Thus, instead of studying problem (\ref{eq:*}) directly, we shall study, for each $k \in \mathbb{N}$, the auxiliary problem
\begin{equation}\label{eq:AP}
\left\{
\begin{aligned}
&{\partial_t^2 u_k -\Delta u_k + f_k(u_k) - \operatorname{div}(a(x) \nabla \partial_t u_k) + \frac{1}{k}\, b(x) \partial_t u_k= 0\quad \hbox{in}\,\,\,\Omega \times (0, +\infty),}\\\
&{u_k=0\quad \hbox{on}\quad \partial \Omega \times (0,+\infty ),}\\\
&{u_k(x,0)=u_{0,k}(x);\quad \partial_t u_k(x,0)=u_{1,k}(x),\quad x\in\Omega,}
\end{aligned}
\right.
\end{equation}
where $f_k: \mathbb{R} \longrightarrow \mathbb{R}$ is defined by
\begin{equation}\label{trunc func}
f_{k}(s):=
\begin{cases}
f(s),  &  |s|\leq k,\\
f(k),  & s> k,\\
f(-k), & s < -k.
\end{cases}
\end{equation}

Here, we use some ideas from Lasiecka and Tataru's work \cite{Lasiecka-Tataru} and
we assume an additional localized frictional damping $b(x)$ satisfying the following assumption:
\begin{assum}
$b\in C^0(\overline{\Omega})$ is a nonnegative function such that $b(x) \geq b_0 >0$ in a neighbourhood of  the boundary $\partial A$ of the set $A:=\{x\in \Omega: a(x)=0\}$, according to Figure \ref{fig1}.
\end{assum}

 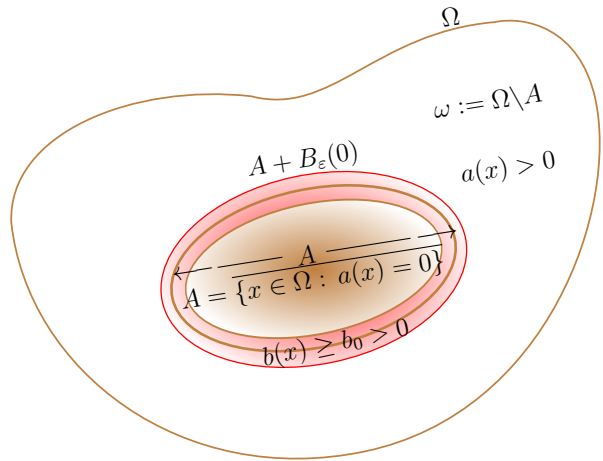
\begin{figure}[H]
\centering
\resizebox{.5\textwidth}{!}{
\begin{tikzpicture}[rotate=12]

\draw[thick,brown,scale=1.2] (-3.5,-0.5) .. controls (-3.3,0.7) and (-1.5,1) .. (0,0.5) .. controls (1,0) and (1.8,1) .. (3.4,0.8) .. controls (3.8,0.8) and (4.5,0.5) .. (4.5,-0.5) .. controls (4.5,-2) and (3,-4.5) .. (0.5,-4.5) .. controls (-2,-4.5) and (-3.7,-2.5) .. (-3.5,-0.5);

\begin{scope}[yshift=-.5cm]
\draw[red,yshift=-.25cm,scale=.8,very thick] (0.5,-2) ellipse (3.1 and 1.9);
\shade[inner color=red,yshift=-.25cm,scale=.8,very thick] (0.5,-2) ellipse (3.1 and 1.9);
\draw[brown,yshift=-.25cm,scale=.8,very thick] (0.5,-2) ellipse (2.6 and 1.3);
\draw[brown,yshift=-.23cm,scale=.8,very thick] (0.5,-2) ellipse (2.6 and 1.3);
\draw[brown,yshift=-.23cm,scale=.8,very thick] (0.5,-2) ellipse (2.9 and 1.6);
\shade[inner color=brown,very thick,scale=.8,yshift=-.3cm]  (0.5,-2) ellipse (2.6 and 1.3);

\node at (.35,-2.0)[rotate=9]{ $A=\overline{\left\{x\in\Omega:\,a(x)=0\right\}}$};
\node at (-1.5,-1.5)[rotate=9]{ $\longleftarrow$};
\node at (-0.58,-1.54)[rotate=9]{ $\overline{\ \ \ \ \ \ \ \ \ \ }$};
\node at (1.34,-1.66)[rotate=9]{ $\overline{\ \ \ \ \ \ \ \ \ \ \ \ \ }$};
\node at (2.45,-1.73)[rotate=9]{ $\longrightarrow$};
\end{scope}
\node at (3.7,-0.3)[rotate=9] {$\omega:=\Omega \backslash A$};
\node  at (3.8,-1.4)[rotate=9] {$a(x)>0$};

\node [rotate=9] at (3.4,1.2) {$\Omega$};

\node[rotate=12] at (.5,-3.6){$b(x) \geq b_0 >0$};

\node[rotate=12] at (0.32,-2.1){$A$};


\node[rotate=9] at (0.6,-0.6) {$A+B_{\varepsilon}(0)$};

\end{tikzpicture}
}
\caption{\small The Kelvin Voigt damping $a(x)$ is positive in $\omega:=\Omega \backslash A$ while the frictional damping $b(x)$ is effective in a neighbourhood of $\partial A$,  that is, $b(x)\geq b_0 >0$ in $V_{\varepsilon}=\{x\in \Omega: d(x,y)  < \varepsilon,~y \in \partial A\}$ for $\varepsilon>0$ small enough. }
\label{fig1}
\end{figure}

In Figures \ref{fig2} and \ref{fig3}, we give examples of subsets of $\Omega$ where the respective Kelvin-Voigt and frictional dissipative functions, $a(x)$ and $b(x)$, are localized.

\begin{figure}[H]
	
	\subfigure{\includegraphics[scale=0.6]{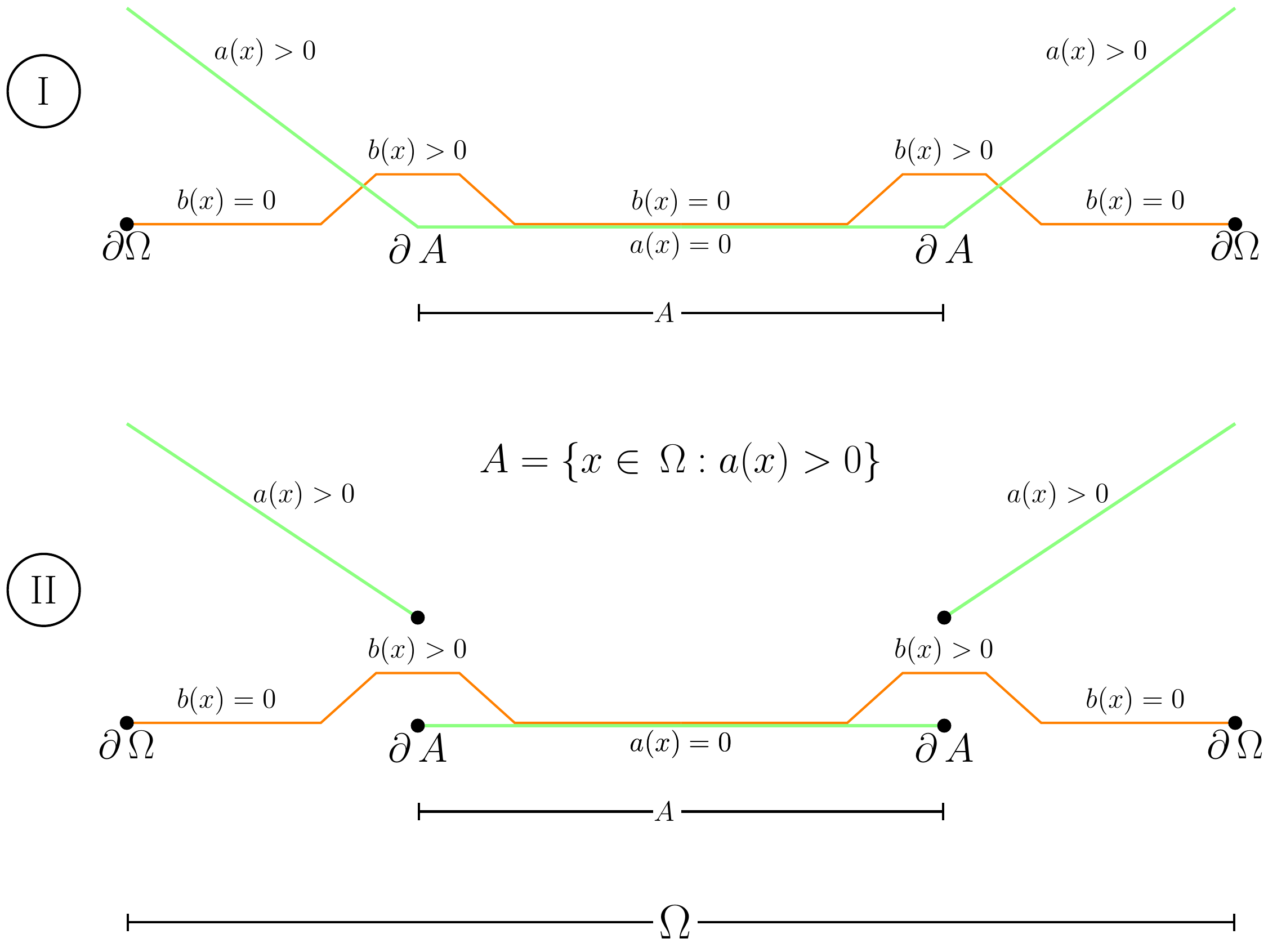}}
	\hfil
	
	\caption{Examples of one-dimensional subsets of $\Omega$ where the respective viscoelastic of Kelvin-Voigt type and frictional dissipative functions, $a(x)$ and $b(x)$,  are localized.}
	\label{fig2}
\end{figure}

\begin{figure}[H]
	
	\subfigure{\includegraphics[width=7cm]{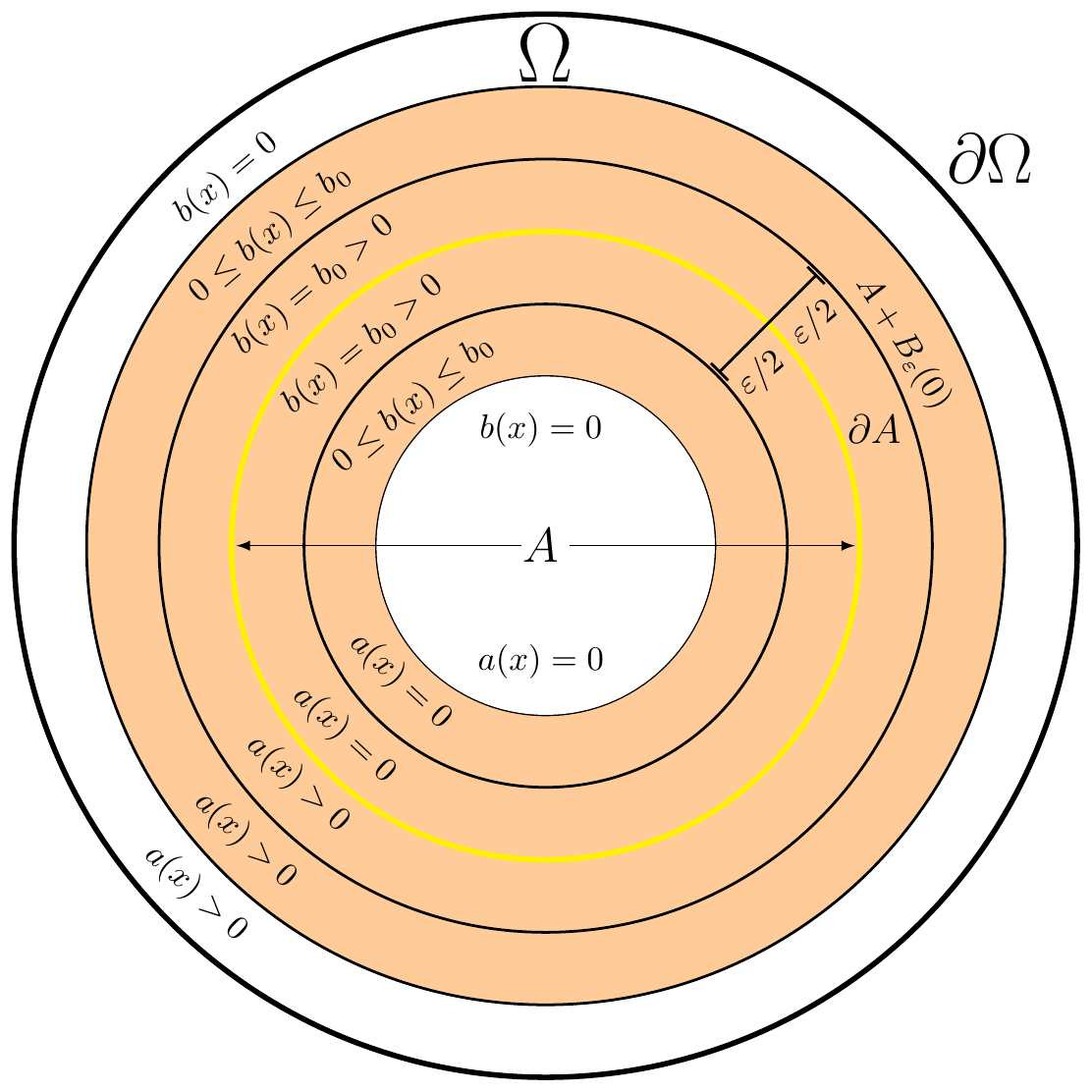}}
	\hfil
	
	\caption{Subsets of $\Omega$ where $a(x)$ and $b(x)$ are localized.}
	\label{fig3}
\end{figure}

In the presence of both dissipative effects, the energy identity associated to problem (\ref{eq:AP}) is given by
\begin{eqnarray}\label{Trunc ener Ident}
E_{u_k}(t) &+& \int_0^t \int_\Omega a(x) |\nabla \partial_t u_k(x,s)|^2\,dxds\\
 &+&\frac{1}{k}\int_0^t\int_\Omega b(x) |\partial_t u_k(x,s)|^2\,dxds = E_{u_k}(0), \hbox{ for all } t\in [0,+\infty) \hbox{ and } k\in \mathbb{N},\nonumber
\end{eqnarray}
where
\begin{eqnarray}\label{Trunc ener}
E_{u_k}(t):= \frac12 \int_{\Omega} |\partial_t u_k(x,t)|^2 + |\nabla u_k(x,t)|^2\,dx + \int_\Omega F_k(u_k(x,t))\,dx,
\end{eqnarray}
with $F_k(s):= \int_0^s f_k(\lambda)\,d\lambda$. Furthermore, we will also prove the corresponding observability inequality to problem (\ref{eq:AP}), that is, we shall prove that there exists a positive constant $C$  which does not depend on $k$, verifying
\begin{equation}\label{Truc Obs Ineq}
E_{u_k}(0) \leq C\,\int_0^{T}\int_{\Omega}\left(\frac{1}{k}\,b(x)|\partial_t u_k|^2 + a(x) |\nabla \partial_t u_k|^2 \right)\,dx\,dt,~\hbox{ for all } T\geq T_0.
\end{equation}

\medskip

In what follows, given $\delta>0$,  $V_{\delta}$ will denote the neighbourhood of $\partial A$ given by
$$V_{\delta}:=\{x\in \Omega: d(x,y)  < \delta,~y \in \partial A\}.$$

It is worth mentioning that if $a(x)=0$, for all $x\in \Omega$, the frictional damping $b(x)$ is not strong enough to provide the exponential and uniform decay of the energy, observe that in this case it violates the geometric control condition (GCC), see references \cite{Bardos}, \cite{Burq-Gerard-CR}. However, the frictional damping plays a key role which we will describe next.

Let $A_\varepsilon:= A + \overline{B_{\varepsilon/2}(0)}$ and $\mathcal{C}:=A\backslash A_\varepsilon=\{x\in A: d(x,y)>\varepsilon/2,~y\in \partial A\}$. To prove (\ref{Truc Obs Ineq}) and therefore the stability result, we will argue by contradiction and we will obtain a sequence $\{u_k^m\}_{m\in \mathbb{N}}$ of solutions to problem (\ref{eq:AP}) such that $ E_{u_k^m}(0)=1$.  The contradiction is obtained proving that $E_{u_k^m}(0)\rightarrow 0$ as $m \rightarrow +\infty$.  By exploiting the properties of $a(\cdot)$ and $b(\cdot)$ and the Poincar\'e inequality, we shall prove that
\begin{eqnarray}\label{convergence}\qquad
\int_0^T \int_{\Omega\backslash \mathcal{C}}  |\partial_t u_k^m|^2 \, dxdt \rightarrow 0,\,\, \hbox{ as } m\rightarrow +\infty.
\end{eqnarray}
To propagate the convergence (\ref{convergence}) to the whole set $\Omega \times (0,T)$, we use microlocal analysis arguments. Indeed, we consider the microlocal defect measure $\mu$, associated to $\{\partial_t u_k^m\}_{m\in \mathbb{N}}$.  First, we shall establish the convergence
\begin{eqnarray*}
\partial_t^2 \partial_t u_k^m - \Delta \partial_t u_k^m \rightarrow 0~ \hbox{ strongly in }~H^{-2}_{loc}(\Omega \times (0,T)), \hbox{ as }m \rightarrow +\infty,
\end{eqnarray*}
 which is enough to ensure that the $\hbox{supp}(\mu)$ is contained in the characteristic set of the wave operator $\{\tau^2=||\xi||^2\}$. However, the last convergence is not sufficient for propagation, since we need a stronger convergence, namely,
 \begin{equation*}
 	\partial_t^2 \partial_t u_k^m - \Delta \partial_t u_k^m \rightarrow 0~ \hbox{ strongly in }~H^{-1}_{loc}(\Omega \times (0,T)), \hbox{ as }m \rightarrow +\infty,
 \end{equation*}
The problematic term, responsible for this lack of regularity, is $ \operatorname{div}(a(x) \nabla \partial_t u_k^m)$ , because
\begin{equation*}
\partial_t \left(\operatorname{div}(a(x) \nabla \partial_t u_k^m) \right)  \rightarrow 0 \hbox{ strongly in } H^{-2}_{loc}(\Omega \times (0,T)), \hbox{ as }m \rightarrow +\infty.
\end{equation*}

To circumvent this difficulty, we use the fact that the frictional damping is effective in the neighbourhood $V_{\varepsilon}$ of $\partial A$.  Note that $a(x) \nabla \partial_t u_k^m=0$ in $A \times (0,T)$, consequently, since
\begin{eqnarray*}
\partial_t^2 u_k^m -\Delta u_k^m= - f_k(u_k^m) - \frac{1}{k}\, b(x) \partial_t u_k^m  ~\hbox{ in } A\times (0,T),.
\end{eqnarray*}
Assuming for a moment that $f_k(u_k^m)\rightarrow 0$ in $L^2(\Omega \times (0,T))$ as $m\rightarrow +\infty$, we obtain
\begin{eqnarray}\label{Dalembert}
\Box \partial_t u_k^m \rightarrow 0 \hbox{ in } H^{-1}_{loc}(\hbox{int}\, A \times (0,T)), \hbox{ as }m \rightarrow +\infty.
\end{eqnarray}

From convergence (\ref{Dalembert}) we deduce that $\mu$ propagates along the bicharacteristic flow of the D'Alembertian operator, which means that if there is $\omega_0=(t_0,x_0,\tau_0,\xi_0)$ such that \linebreak $\omega_0=(t_0,x_0,\tau_0,\xi_0) \notin \hbox{supp}(\mu)$, the whole bicharacteristic issued from $\omega_0$ does not belong to $\hbox{supp}(\mu)$. However, since $\hbox{supp}(\mu) \subset \mathcal{C} \times (0,T)\subset A \times (0,T)$, and the frictional damping acts in both sides of the boundary $\partial A$, we can propagate the kinetic energy from $(V_{ {\varepsilon} /2} \cap A) \times (0,T)$ towards the set $\mathcal{C} \times (0,T)$.

Consequently,
\begin{eqnarray}\label{propagation}
\int_0^T \int_\Omega |\partial_t  u_k^m(x,t)|^2\,dxdt \rightarrow 0, ~\hbox{ as }m \rightarrow + \infty.
\end{eqnarray}

In light of convergence (\ref{propagation}) and an argument of equipartition of energy, we can conclude that $E_{u_k^m}(0)\rightarrow 0$, as $m \rightarrow +\infty$, obtaining the desired contradiction. This will be clarified in section 3.

 The next step in the proof is to remove the frictional damping term. For this purpose we are going to prove that the sequence $\{u_k\}_{k\in \mathbb{N}}$ of solutions to problem (\ref{eq:AP}) converges to the unique solution of problem  (\ref{eq:*}). In addition, passing to the limit in (\ref{Trunc ener Ident}) and (\ref{Truc Obs Ineq}), we achieve the energy identity and  the observability inequality associated to problem (\ref{eq:*}), respectively, which are the necessary and sufficient ingredients to establish its exponential stability result.

 \section{Convergence of the auxiliary problem}
\setcounter{equation}{0}

\medskip
\subsection{The limit process.}
\medskip

In this section we prove that the sequence $\{u_k\}_{k\in \mathbb{N}}$ of solutions to problem (\ref{eq:AP}) converges to the unique solution to the problem (\ref{eq:*}).

The function $f$ satisfies the following hypotheses:
\begin{assum}\label{assumption 2.1}
$f:\mathbb{R}\rightarrow \mathbb{R}$ is a $C^2$ function with sub-critical growth; satisfying the sign condition $f(s) s \geq 0$, for all $s\in \mathbb{R}$, and
\begin{eqnarray}\label{main ass on f}
f(0)=0, \ \ \ |f^{(j)}(s)| \leq k_{0}(1 + |s|)^{p-j}, \hbox{ for all } s \in \mathbb{R} \hbox{ and } j=1,2.
\end{eqnarray}
In particular, we obtain from (\ref{main ass on f}),
\begin{eqnarray}\label{ass on f}
|f(r) - f(s)| \leq c \left(1 + |s|^{p-1} + |r|^{p-1} \right)|r-s|,~\hbox{ for all } s,r\in \mathbb{R},
\end{eqnarray}
for some $c>0$, with
\begin{eqnarray}\label{ass on f'}
1\leq p \leq \frac{n+2}{n-2}~\hbox{ if }n\geq 3~\hbox{ or } ~p\geq 1~\hbox{ if }~ n=1,2.
\end{eqnarray}

In addition,
\begin{eqnarray}\label{ass on F}
0 \leq F(s) \leq f(s) s,~\hbox{ for all } s\in \mathbb{R},
\end{eqnarray}
where $F(\lambda):= \int_0^\lambda f(s)\,ds$.
\medskip
\end{assum}

We begin with some preliminary results.
\begin{Lemma}\label{Lemma1}
The distributional derivative $f_k'$ of the function defined in (\ref{trunc func}) is the essentially bounded function $g_k:\mathbb{R} \rightarrow \mathbb{R}$ given by
\begin{equation}\label{drrivative}
g_k(s):=
\begin{cases}
f'(s), & |s|\leq k,\\
0, & s> k,\\
0, & s < -k.
\end{cases}
\end{equation}
\end{Lemma}
\begin{proof}
Take $\varphi \in C_0^\infty(\mathbb{R})$. Once $f_k\in L^1_{loc}(\mathbb{R})$ we have
\begin{eqnarray*}
&&\left<f_k',\varphi \right>_{\mathcal{D}'(\mathbb{R}), \mathcal{D}(\mathbb{R})}= - \int_{\mathbb{R}} f_k(s) \varphi'(s)\,ds\\
&&= -\left[\int_{-\infty}^{-k} f_k(s) \varphi'(s)\,ds + \int_{-k}^k f_k(s) \varphi'(s)\,ds  + \int_{k}^{+\infty} f_k(s) \varphi'(s)\,ds  \right]\\
&&= - \left[ f(-k) \varphi(-k)  + f(k) \varphi(k) - f(-k) \varphi(-k) - \int_{-k}^k f'(s) \varphi(s)\,ds - f(k)\varphi(k)\right]\\
&& =  \int_{-k}^k f'(s) \varphi(s)\,ds = \int_{\mathbb{R}} g(s) \varphi(s) \,ds.
\end{eqnarray*}
\end{proof}

Consider the following result which will be useful to the proof of Lemma \ref{Lema2}.
\begin{Theorem}\label{brezis}
	Let $u \in W^{1,p}(I)$ with $1 \leq p \leq \infty$, where $I$ is a bounded interval of $\mathbb{R}$. Then, there exists $\widetilde{u} \in C(\bar{I})$ such that
	\begin{equation*}
	u=\widetilde{u} \hbox{ a.e. in } I
	\end{equation*}
	and
	\begin{equation*}
	\widetilde{u}(x)-\widetilde{u}(y)=\int_{y}^{x} u'(t)dt \hbox{ for all } x, y \in \bar{I}.
	\end{equation*}
\end{Theorem}
\begin{proof}
See Brezis \cite{Brezis}, Theorem 8.2.
\end{proof}
\begin{Lemma}\label{Lema2}
For each $k\in \mathbb{N}$, there exists a positive constant $C_k$ verifying
\begin{equation*}
|f_k(r)-f_k(s)|\leq C_k|r-s| \hbox{ for every} r,s \in  \mathbb{R},
\end{equation*}
where $f_k$ is the function defined in \eqref{trunc func}.
\end{Lemma}
\begin{proof}
Consider $s, r\in \mathbb{R}$ with $s< r$. Applying Theorem \ref{brezis} for $I=]s,r[$,  it follows that
\begin{eqnarray*}
f_k(r) - f_k(s) = \int_s^r f_k'(\xi)\,d\xi.
\end{eqnarray*}
Thus, Lemma \ref{Lemma1} yields the following inequality:
\begin{eqnarray}\label{Lipschitz}
|f_k(r) - f_k(s)| \leq \int_s^r |f_k'(\xi)|\,d\xi \leq \sup_{s\in [-k,k]}|g_k(s)|\,|r-s|,
\end{eqnarray}
which concludes the proof.
\end{proof}

From Lemma \ref{Lema2}, for each $k\in \mathbb{N}$, standard arguments of Semigroup theory yield that problem (\ref{eq:AP}) possesses an unique regular solution $u_k$ in the class
\begin{eqnarray*}
C^1([0,\infty); H_0^1(\Omega)) \cap C^2([0,\infty); L^2(\Omega)),
\end{eqnarray*}
and, furthermore, the map $t \mapsto \operatorname{div}(\nabla u(t) + a(x) \partial_t \nabla u(t))\in L^2(\Omega)$ is continuous in $[0,\infty)$.

Multiplying the first equation of (\ref{eq:AP}) by $\partial_t u_k$ and performing integration by parts, it yields
\begin{eqnarray}\label{est1}
&&\frac12 \frac{d}{dt} ||\partial_t u_k(t)||_{L^2(\Omega)}^2 + \frac12 \frac{d}{dt} ||\nabla u_k(t)||_{L^2(\Omega)}^2 + \frac{d}{dt}\int_\Omega F_k(u_k(x,t))\,dxdt\\
&& + \int_\Omega a(x) |\nabla \partial_t u_k(x,t)|^2\,dx +\frac{1}{k} \int_\Omega b(x) |\partial_t u_k(x,t)|^2\,dx=0, \hbox{ for all } t\in [0,\infty),\nonumber
\end{eqnarray}
where
\begin{eqnarray}\label{primitive}
F_k(\lambda) = \int_0^\lambda f_k(s)\,ds.
\end{eqnarray}

Hence, taking (\ref{est1}) into account, we infer
\begin{eqnarray}\label{est2}
E_{u_k}(t) &+& \int_0^t \int_\Omega a(x) |\nabla \partial_t u_k(x,s)|^2\,dxds\\
 &+&\frac{1}{k}\int_0^t\int_\Omega b(x) |\partial_t u_k(x,s)|^2\,dxds = E_{u_k}(0), \hbox{ for all } t\in [0,+\infty) \hbox{ and } k\in \mathbb{N},\nonumber
\end{eqnarray}
where
\begin{eqnarray}
E_{u_k}(t):= \frac12 \int_\Omega  |\partial_t u_k(x,t)|^2 + |\nabla u_k(x,t)|^2\,dx + \int_\Omega F_k(u_k(x,t))\,dx,
\end{eqnarray}
is the energy associated to problem (\ref{eq:AP}).

We observe that from (\ref{trunc func}), the function defined in (\ref{primitive}) is given by
\begin{equation}\label{primitive trunc func}
F_{k}(s):=\begin{cases}
\displaystyle \int_0^sf(\xi)\,d\xi, & |s|\leq k,\\
\displaystyle \int_0^k f(\xi) \,d\xi + f(k)[s-k], & s > k,\\
\displaystyle f(-k)[s+k] + \int_0^{-k}f(\xi)\,d\xi, & s < -k.
\end{cases}
\end{equation}

Since $f$ satisfies the sign condition, it results that $F_k(s) \geq 0$ for all $s\in \mathbb{R}$ and $k\in \mathbb{N}$. In addition, from (\ref{ass on f}) and (\ref{ass on F}), we obtain, respectively, that $|f(s)|\leq c [|s| + |s|^p]$ and $0 \leq F(s) \leq f(s) \,s$ for all $s\in \mathbb{R}$.  Then, we infer that
\begin{eqnarray}\label{bound on Fk}
|F_k(s)| \leq c [|s|^2 + |s|^{p+1}],~\hbox{ for all }s\in \mathbb{R}~\hbox{ and }~k\in\mathbb{N}.
\end{eqnarray}

Consequently,
\begin{eqnarray}\label{bound data L1}
\int_\Omega |F_k(u_{0,k})|\,dx &\leq& c\int_\Omega \left[|u_{0,k}|^2 + |u_{0,k}|^{p+1} \right]\,dx\\
&\lesssim&  ||u_{0,k}||_{H_0^1(\Omega)}.\nonumber
\end{eqnarray}

Assuming that $p\geq 1$ is under conditions (\ref{ass on f'}), we have for every dimension $n\geq 1$ that $H_0^1(\Omega) \hookrightarrow L^{p+1}(\Omega)$, which implies that the RHS of  (\ref{bound data L1}) is bounded.
So, estimates (\ref{est2}) (also called energy identity for the auxiliary problem (\ref{eq:AP}) and (\ref{bound data L1}) and convergence (\ref{conv init data}),  yield a subsequence of $\{u_k\}$, reindexed again by $\{u_k\}$,  such that
\begin{eqnarray}
&&u_k \rightharpoonup u ~\hbox{ weakly * in } L^\infty(0,\infty; H_0^1(\Omega)),\label{conver1}\\
&&\partial_t u_k \rightharpoonup \partial_t u ~\hbox{ weakly * in } L^\infty(0,\infty; L^2(\Omega))\label{conver2},\\
&&  \frac{1}{\sqrt{k}}\sqrt{b(x)}\partial_t u_{k} \rightarrow  0 \hbox{ strongly in }  L_{loc}^{2}(0,\infty;
L^{2}(\Omega)),\label{eq:3.54}\\
&& \sqrt{a(x)} \nabla \partial_t u_{k} \rightharpoonup \sqrt{a(x)} \nabla \partial_t u \hbox{ weakly in }L^{2}(0,\infty; L^{2}(\Omega))~ (\hbox{see Appendix)}.\label{eq:3.56}
\end{eqnarray}

Employing the standard compactness result (see Simon \cite{Simon}) we also deduce that
\begin{eqnarray}\label{conver3}
u_k \rightarrow u~\hbox{ strongly in } L^\infty (0,T; L^{2^{\ast}-\eta}(\Omega)); \hbox{ for all } T>0,
\end{eqnarray}
where $2^{\ast}:= \frac{2n}{n-2}$ and $\eta >0$ is small enough. In addition, from (\ref{conver3}), we obtain
\begin{eqnarray}\label{conver4}
u_k \rightarrow u~\hbox{ a. e. in  } \Omega \times (0,T), \hbox{ for all } T>0.
\end{eqnarray}

On the other hand, from (\ref{ass on f}), (\ref{ass on f'}), (\ref{conver1}) and once $H_0^1(\Omega)\hookrightarrow  L^{p+1}(\Omega)\hookrightarrow  L^{\frac{p+1}{p}}(\Omega)$
the following estimate holds:
\begin{eqnarray}\label{estIII}
\|f_k(u_k)\|_{L^{\frac{p+1}{p}}}^{\frac{p+1}{p}} &=& \int_0^{T} \int_\Omega |f_k(u_k(x,t))|^{\frac{p+1}{p}}\,dxdt \nonumber\\
&\lesssim& \int_0^T\int_\Omega |u_k|^{\frac{p+1}{p}}\,dxdt + \int_{0}^{T}\int_\Omega |u_k|^{p+1}\,dx dt\nonumber\\
&=& \int_0^T \|u_k\|_{L^{\frac{p+1}{p}}(\Omega) }^{\frac{p+1}{p}}\,dt+\int_0^T \|u_k\|_{L^{p+1}(\Omega)}^{p+1}\, dt\nonumber\\
&\lesssim& \int_0^T \|u_k\|_{H_0^1(\Omega)}^{\frac{p+1}{p}}\,dt+\int_0^T \|u_k\|_{H_0^1(\Omega)}^{p+1}\, dt\nonumber\\
&\lesssim& \|u_k\|_{L^\infty(0,T;H_0^1(\Omega))}^{\frac{p+1}{p}}+\|u_k\|_{L^\infty(0,T;H_0^1(\Omega))}^{p+1}\nonumber\\
& \leq& c <+\infty,~\hbox{ for all }t\geq 0.
\end{eqnarray}
It is easy to see that
\begin{equation}\label{estIII.1}
f(u) \in L^\infty(0,\infty;L^{\frac{p+1}{p}}(\Omega)).
\end{equation}
Indeed,
\begin{eqnarray}\label{estIII.2}
\int_\Omega |f(u(x,t))|^{\frac{p+1}{p}}\,dx &\lesssim& \int_\Omega |u(x,t)|^{\frac{p+1}{p}}\,dx + \int_\Omega |u(x,t)|^{p+1}\,dx \nonumber\\
& \lesssim& \|u(\cdot,t)\|_{H_{0}^{1}(\Omega)}^{\frac{p+1}{p}}+ \|u(\cdot,t)\|_{H_{0}^{1}(\Omega)}^{p+1}\nonumber\\
& \lesssim& \|u\|_{L^\infty(0,T;H_{0}^{1}(\Omega))}^{\frac{p+1}{p}} +\|u\|_{L^\infty(0,T;H_{0}^{1}(\Omega))}^{p+1}<+\infty,~\hbox{ for all }t\geq 0.
\end{eqnarray}
From \eqref{estIII.2} and the definition of essential supremum we obtain \eqref{estIII.1}.

In addition, from (\ref{conver4}) and the continuity of the function $f$, we get
\begin{eqnarray}\label{conver5''}
f_k(u_k) \rightarrow f(u) \hbox{ a. e. in  } \Omega \times (0,T), \hbox{ for all } T>0.
\end{eqnarray}

Indeed, the convergence (\ref{conver4}) guarantees the existence of set $Z_T \subset \Omega \times (0,T)$ with $\operatorname{meas}(Z_T)=0$  such that $u_k(x,t) \rightarrow u(x,t)$ for all $(x,t) \in \Omega \times (0,T) \setminus Z_T$ when $k \rightarrow \infty$. Therefore, for all $(x,t) \in \Omega \times (0,T) \setminus Z_T$ there exists a positive constant $L=L(x,t)>0$ verifying $|u_k(x,t)|<L,$ for all $k \in \mathbb{N}$.  Then, using the definition of $f_k$, we obtain that
\begin{eqnarray}\label{boundedness}
\hbox{ if } |u_k(x,t)|<L, \hbox{ for all } k \in \mathbb{N} \hbox{ then }~ f_k(u_k(x,t))=f(u_k(x,t)), \hbox{ for all } k\geq L,
\end{eqnarray}
that is,
\begin{equation}\label{boundednes1}
f_k(u_k(x,t))-f(u_k(x,t)) \rightarrow 0 \hbox{ when } k \rightarrow \infty \hbox{ for all } (x,t) \in \Omega \times (0,T) \setminus Z_T.
\end{equation}
On the other hand, employing the continuity of $f$ it follows that
\begin{equation}\label{boundednes2}
f(u_k(x,t))-f(u(x,t)) \rightarrow 0 \hbox{ when } k \rightarrow \infty \hbox{ for all } (x,t) \in \Omega \times (0,T) \setminus Z_T.
\end{equation}
From \eqref{boundednes1} and \eqref{boundednes2} the convergence (\ref{conver5''}) holds.


\begin{Lemma}[Strauss] Let $\mathcal{O}$ be an open and bounded subset of $\mathbb{R}^N$, $N\geq 1$, $1<q<+\infty$ and $\{u_n\}_{n\in \mathbb{N}}$ a sequence which is bounded in $L^q(\mathcal{O})$. If $u_n \rightarrow u$ a.e. in $\mathcal{O}$, then $u\in L^q(\mathcal{O})$ and $u_n \rightharpoonup u$ weakly in $L^q(\mathcal{O})$. In addition, if $1\leq r <q$ we also have $u_n \rightarrow u$ strongly in $L^r(\mathcal{O})$.
\end{Lemma}
\begin{proof}
	See  \cite{Brezis}, Exercise 4.16 or \cite{Strauss}.
\end{proof}

Gathering together \eqref{estIII}, \eqref{estIII.1} and Lions' Lemma, we deduce that \begin{eqnarray}\label{weak conv fk}
f_k(u_k) \rightharpoonup f(u)~\hbox{ weakly in } L^{\frac{p+1}{p}}(\Omega \times (0,T)).
\end{eqnarray}

Going back to problem (\ref{eq:AP}), multiplying by $ \varphi \, \theta $, where $ \varphi \in C_0^\infty(\Omega), \theta \in C_0^\infty(0,T)$ and performing integration by parts, we obtain
 \begin{eqnarray}\label{form var}
&&-\int_0^T \theta'(t) \int_\Omega \partial_t u_{k}(x,t)\, \varphi(x) \,dx dt + \int_0^T \theta(t) \int_\Omega \nabla u_{k}(x,t)\cdot \nabla \varphi(x) \,dxdt\\
&& +\int_0^T \theta(t) \int_\Omega f_k(u_k(x,t))\,\varphi(x) \,dxdt +\int_0^T \theta(t) \int_\Omega a(x)  \nabla \partial_t u_k(x,t)\cdot \nabla \varphi(x)\,dxdt\nonumber\\
&&+\frac{1}{k}\int_0^T \theta(t) \int_\Omega b(x)  \partial_t u_k(x,t)\, \varphi(x)\,dxdt\nonumber = 0.
 \end{eqnarray}

Passing to the limit in (\ref{form var}) and observing convergences (\ref{conver1})-(\ref{eq:3.56}) and (\ref{weak conv fk}), we get
 \begin{eqnarray}\label{limit}
&&-\int_0^T \theta'(t) \int_\Omega \partial_t u(x,t)\, \varphi(x) \,dx dt + \int_0^T \theta(t) \int_\Omega \nabla u(x,t)\cdot \nabla \varphi(x) \,dxdt\\
&& +\int_0^T \theta(t) \int_\Omega f(u(x,t))\,\varphi(x) \,dxdt + \int_0^T \theta(t) \int_\Omega a(x)  \nabla \partial_t  u(x,t) \cdot \nabla \varphi(x)\,dxdt\nonumber = 0,
 \end{eqnarray}
 for all $\varphi \in C_0^\infty(\Omega)$ and $\theta \in C_0^\infty(0,T)$.  We conclude that
 \begin{eqnarray}\label{dist sol}
 \partial_t^2 u - \Delta u + f(u) - \operatorname{div}(a(x)\nabla \partial_t u) = 0 ~\hbox{ in }\mathcal{D}'(\Omega \times (0,T)),
 \end{eqnarray}
 and since
 \begin{eqnarray*}
&& a(\cdot) \partial_t u \in L^\infty(0,T; L^2(\Omega)), ~\Delta u \in L^\infty(0,T; H^{-1}(\Omega)), \\
&&\operatorname{div}( a(x)\nabla  \partial_t u) \in L^2(0,T; H^{-1}(\Omega)) ~\hbox{and}~ f(u) \in L^\infty(0,T; L^{\frac{p+1}{p}}(\Omega)),
 \end{eqnarray*}
we deduce that $\partial_t^2 u \in  L^2(0,T; H^{-1}(\Omega))$ and
\begin{eqnarray}\label{weak solution'}
 \partial_t^2 u - \Delta u + f(u) - \hbox{div} \left[a(x) \partial_t u\right] = 0 ~\hbox{ in } L^2(0,T;  H^{-1}(\Omega)).
\end{eqnarray}

Applying Lemma 8.1 of Lions-Magenes \cite{Lions-Magenes}, we deduce that
\begin{equation}\label{weak solution}
u\in C_w(0,T;H_0^1(\Omega)) \hbox{ and } \partial_t u \in C_w(0,T; L^2(\Omega)),
\end{equation}
where $C_w(0,T;Y)=$ space of functions $f\in L^\infty(0,T;Y)$ whose mappings $[0,T] \mapsto Y$ are weakly continuous, that is, $t \mapsto \langle y', f(t) \rangle_{Y',Y}$ is continuous in $[0,T]$ for all $y' \in Y'$, dual of $Y$.

\medskip

Our first result reads as follows:
 \begin{Theorem}\label{theo 1}
 Assume that $a\in L^\infty(\Omega)\cap C^0(\overline{\omega})$  and $f\in C^1(\mathbb{R})$ satisfies $f(s)s\geq 0$ for all $s\in \mathbb{R}$.  In addition, suppose that assumptions (\ref{ass on f}), (\ref{ass on f'}) and (\ref{ass on F}) are in place.  Then, problem (\ref{eq:*}) has at least a global solution in the class
 {\small
 $$u\in C_w(0,T;H_0^1(\Omega)),~\partial_t u \in C_w(0,T; L^2(\Omega)),~\partial_t^2 u \in L^2(0,T; H^{-1}(\Omega)),$$}
 provided that $\{u_0,u_1\}\in H_0^1(\Omega) \times L^2(\Omega)$. Furthermore, assuming that $1\leq p \leq \frac{n}{n-2},n\geq 3$ or $p\geq 1, n=1,2$, we have the uniqueness of solution.
 \end{Theorem}

\medskip
 \begin{proof}
The uniqueness of solution as well as to prove that $u(0)=u_0$ and $\partial_t u(0)=u_1$ follow the same ideas used in Lions \cite{Lions1} (Theorem 1.2).
\end{proof}

\medskip

\subsection{Recovering the regularity in time for the range $1\leq p < \frac{n}{n-2},$ $n\geq 3$. }

When $p\geq 1,$ $n=1,2$, the result is trivially verified and it will be omitted.

The goal of this subsection is to prove that if $1\leq p < \frac{n}{n-2},$ $n\geq 3$, the related solutions to problem (\ref{eq:*}) are in the class
$$u\in C^0([0,T];H_0^1(\Omega)),~\partial_t u \in C^0([0,T]; L^2(\Omega))$$
and, in addition, one has
\begin{eqnarray*}
\{u_k, \partial_t u_k\} \rightarrow \{u,\partial_tu\} \hbox{ in } C^0([0,T];H_0^1(\Omega)) \times C^0([0,T]; L^2(\Omega)).
\end{eqnarray*}

To prove the above statements, we need to prove that
\begin{eqnarray}\label{main conv}
f_k(u_k) \rightarrow f(u) ~\hbox{ strongly in }L^2(\Omega \times (0,T)).
\end{eqnarray}

In fact, first we observe that
\begin{eqnarray}\label{I}
&&\int_0^T \int_\Omega |f_k(u_k) - f(u)|^2\,dxdt\\
&& \lesssim \int_0^T \int_\Omega |f_k(u_k) - f(u_k)|^2\,dxdt + \int_0^T \int_\Omega |f(u_k) - f(u)|^2\,dxdt.\nonumber
\end{eqnarray}

In view of (\ref{ass on f}) one has
\begin{eqnarray*}
\int_\Omega |f(u_k) - f(u)|^2 \,dx &\lesssim& \int_\Omega |u_k - u|^2\,dx + \int_\Omega |u_k|^{2(p-1)} |u_k -u|^2\,dx \nonumber \\
&&+ \int_\Omega |u|^{2(p-1)} |u_k -u|^2\,dx \nonumber \\
&=& I_{1,k}+I_{2,k}+I_{3,k}
\end{eqnarray*}

We observe that since $\frac{p-1}{p} + \frac{1}{p}=1$, H\"older inequality yields
\begin{eqnarray*}
I_{2,k} \leq \left(\int_\Omega |u_k|^{2p} \right)^{\frac{p-1}{p}}\left( \int_\Omega |u_k - u|^{2p}\right)^{\frac{1}{p}}.
\end{eqnarray*}

Choosing $p< \frac{n}{n-2}$ it implies that $2p < \frac{2n}{n-2}=2^{\ast}$ and, consequently, from (\ref{conver1}) and (\ref{conver3}) we deduce that $I_{2,k} \rightarrow 0$ as $k\rightarrow +\infty$. Analogously, we also deduce that $I_{3,k} \rightarrow 0$ as $k\rightarrow +\infty$.  We trivially obtain that $I_{1,k} \rightarrow 0$ as $k\rightarrow +\infty$.  Then,
\begin{eqnarray}\label{II}
\int_0^T \int_\Omega |f(u_k) - f(u)|^2\,dxdt \rightarrow 0 ~\hbox{ as } k \rightarrow \infty.
\end{eqnarray}

From (\ref{I}) it remains to prove that
\begin{eqnarray}\label{III}
\int_0^T \int_\Omega |f_k(u_k) - f(u_k)|^2 \,dxdt \rightarrow 0 ~\hbox{ as } k \rightarrow \infty.
\end{eqnarray}

Let us consider, initially, $t\in [0,T]$ fixed and define $$\Omega_k^t:=\{x\in \Omega: |u_k(x,t)| >k\}.$$

Observing that
\begin{eqnarray*}
f_k(u_k) - f(u_k) =0, ~\hbox{ if } |u_k(x,t)|\leq k,
\end{eqnarray*}
we have

\begin{eqnarray}\label{IV}
&&\int_\Omega |f_k(u_k) - f(u_k)|^2\,dx = \int_{\Omega_k^t} |f_k(u_k) - f(u_k)|^2\,dx\\
&&\lesssim \left[\int_{\Omega_k^t} |f(u_k)|^2\, dx + \int_{\Omega_k^t} |f(-k)|^2\, dx + \int_{\Omega_k^t} |f(k)|^2\, dx  \right]\nonumber\\
&& \lesssim\left[ \int_{\Omega_k^t} [|u_k|^2 + |u_k|^{2p} ]\, dx +  \int_{\Omega_k^t} [|k|^2 + |k|^{2p} ]\, dx\right]\nonumber\\
&& \lesssim\left[ \int_{\Omega_k^t} |u_k|^{2p} \, dx +  \int_{\Omega_k^t} |k|^{2p}\, dx\right]\nonumber\\
&& \lesssim \int_{\Omega_k^t} |u_k|^{2p} \,dx.\nonumber
\end{eqnarray}

Before analyzing the term on the RHS of (\ref{IV}) we note that since $H_0^1(\Omega)\hookrightarrow L^{\frac{2n-\frac12}{n-2}}(\Omega)$ and the convergence (\ref{conv init data}) are in place, we obtain
\begin{eqnarray} \label{V}
\left( \int_{\Omega_k^t} k^{\frac{2n-\frac{1}{2}}{n-2}}\,dx \right) &\lesssim& \left( \int_{\Omega_k^t}
|u_k|^{\frac{2n-\frac{1}{2}}{n-2}}\,dx \right)\\
&=& ||u_k(t)||_{L^{\frac{2n-\frac12}{n-2}}(\Omega_k^t)}^{\frac{2n-\frac12}{n-2}} \lesssim ||u_k(t)||_{H_0^1(\Omega)}^{\frac{2n-\frac12}{n-2}} \lesssim [E_{u_k}(0)]^{\frac{2n-\frac12}{n-2}} \leq C,\nonumber
\end{eqnarray}
for all $t\in [0,T]$, where $C$ is a positive constant which does not depend on $k$ and $t$.
Thus, it yields
\begin{eqnarray}\label{VI}
\operatorname{meas}(\Omega_k^t)  \lesssim k^{\frac{-2n+\frac12}{n-2}}, \hbox{ for all } t\in [0,T].
\end{eqnarray}

\medskip

Let $\beta:= \frac{2n}{(2p)(n-2)}$, for $n\geq 3$. Observe that we have the following inequalities:
$$p< \frac{n}{n-2} \Leftrightarrow  2n> (2p)(n-2)\Leftrightarrow 2p< \frac{2n}{n-2}=2^{\ast} \Leftrightarrow  \beta > 1. $$

Setting $\alpha>0$ such that $\frac{1}{\alpha} + \frac{1}{\beta}=1$, we deduce that $\alpha=\frac{2n}{2n-(2p)(n-2)}$ and using H\"older inequality we get
\begin{eqnarray}\label{VII'}
\int_{\Omega_k^t} |u_k|^{2p}\,dx &\leq& \left( \operatorname{meas}(\Omega_k)\right)^{\frac{2n-(2p)(n-2)}{2n}} \left(\int_{\Omega_k^t} |u_k|^{\frac{2n}{n-2}}\right)^{\frac{(2p)(n-2)}{2n}}\\
&=& \left( \operatorname{meas}(\Omega_k)\right)^{\frac{2n-(2p)(n-2)}{2n}} ||u_k(t)||_{L^{\frac{2n}{n-2}}(\Omega)}^{2p}.\nonumber
\end{eqnarray}

Thus, from (\ref{VI}) and (\ref{VII'}) we conclude
\begin{eqnarray}\label{VIII'}
\int_0^T\int_{\Omega_k^t} |u_k|^{2p}\,dx &\leq& k^{\left(\frac{-2n+\frac12}{n-2}\right)\left(\frac{2n-(2p)(n-2)}{2n}\right)} \int_0^T||u_k(t)||_{L^{\frac{2n}{n-2}}(\Omega)}^{2p}\,dt\\
&\lesssim& k^{\left(\frac{-2n+\frac12}{n-2}\right)\left(\frac{2n-(2p)(n-2)}{2n}\right)}\int_0^T ||u_k(t)||_{H_0^1(\Omega)}^{2p}\,dt\nonumber\\
&\lesssim& k^{\left(\frac{-2n+\frac12}{n-2}\right)\left(\frac{2n-(2p)(n-2)}{2n}\right)}[ E_{u_k}(0)]^p,\nonumber
\end{eqnarray}

Employing the fact that $E_{u_k}(0) \leq C$ for all $k\in \mathbb{N}$ and $\left(\frac{-2n+\frac12}{n-2}\right)\left(\frac{2n-(2p)(n-2)}{2n}\right)< 0$, in light of inequality (\ref{VIII'}) , we prove that

\begin{eqnarray}\label{VIII'.convergente}
\int_0^T\int_{\Omega_k^t} |u_k|^{2p}\,dx \rightarrow 0,\, \hbox{ as } k\rightarrow +\infty.
\end{eqnarray}

Gathering (\ref{IV}) and (\ref{VIII'.convergente}) together, we conclude (\ref{III}) which proves (\ref{main conv}).

\medskip

Now, we define the sequence $z_{\mu,\sigma}=u_{\mu}-u_{\sigma}$, $\mu,\sigma\in
\mathbb{N}$, and from (\ref{eq:AP}) we deduce
\begin{eqnarray}\label{eq:3.59}
&&\frac12 \frac{d}{dt} \left\{||\partial_t z_{\mu,\sigma}(t)||_{L^2(\Omega)}^2 +||\nabla
z_{\mu,\sigma}(t)||_{L^2(\Omega)}^2 \right\}
+\int_{\Omega} a(x)|\partial_t \nabla z_{\mu,\sigma}|^2\,dx \\
&&+ \frac1{\mu} \int_{\Omega}b(x) |u_{\mu}'|^2\,dx - \frac1{\mu}
\int_{\Omega} b(x) u_{\mu}' u_{\sigma}' \,dx - \frac1{\sigma}
\int_{\Omega} b(x)  u_{\sigma}' u_{\mu}' \,dx +  \frac1{\sigma}
\int_{\Omega} b(x) |u_{\sigma}'|^2dx\nonumber\\
&&=  \int_{\Omega}\left(f_{\mu}(u_{\mu})-
f_{\sigma}(u_{\sigma})\right)(\partial_t u_{\mu}-\partial_t u_{\sigma}) \,dx.
\nonumber
\end{eqnarray}

Integrating (\ref{eq:3.59}) over $(0,t)$, we obtain
\begin{eqnarray}\label{eq:3.60}
&&\frac12  \left\{||\partial_t z_{\mu,\sigma}(t)||_{L^2(\Omega)}^2 +||\nabla
z_{\mu,\sigma}(t)||_{L^2(\Omega)}^2 \right\}
+\int_0^t\int_{\Omega} a(x)|\nabla\partial_t z_{\mu,\sigma}|^2\,dx ds\\
&& \lesssim \left[\frac1{\mu}+\frac1{\sigma}\right]\int_0^t\int_{\Omega}
b(x) |u_{\mu}'|^2\,dx
+ \left[\frac1{\mu}+\frac1{\sigma}\right]\int_0^t\int_{\Omega} b(x) |u_{\sigma}'|^2\,dx\nonumber\\
&&+  \frac12\left\{||u_{1,\mu} - u_{1,\sigma}||_{L^2(\Omega)}^2 + ||\nabla u_{0,\mu}
- \nabla u_{0,\sigma}||_{L^2(\Omega)}^2\right\}\nonumber\\
&&+ \int_0^t \int_{\Omega}\left(f_{\mu}(u_{\mu})-
f_{\sigma}(u_{\sigma})\right)(\partial_t u_{\mu}-\partial_t u_{\sigma}) \,dxds.
\nonumber
\end{eqnarray}

Since $b\in C^0(\overline{\Omega})$, the convergences (\ref{conv init data}), (\ref{conver2}) and (\ref{main conv}) imply that the terms on the RHS of  the  (\ref{eq:3.60}) converges to zero
as $\mu, \sigma \rightarrow +\infty$.  Thus, we deduce that
\begin{eqnarray}\label{Cauchy conv}\quad
&&u_{\mu} \rightarrow u \hbox { in }
C^0([0,T];H_{0}^1(\Omega)) \cap C^1 ([0,T]; L^2(\Omega)),\\
&&\lim_{\mu \rightarrow +\infty}
\int_0^T\int_{\Omega} a(x) |\nabla \partial_t u_{\mu}|^2
dx\, ds = \int_0^T\int_{\Omega} a(x) |\nabla \partial_t u|^2
dx\, ds,~(\hbox{\small see the Appendix}),\label{damping conv'}
\end{eqnarray}
for all $T>0$.

\medskip

\subsection{Estimating $F_k(u_k)$ }
\medskip

Inequality (\ref{bound on Fk}) gives $$|F_k(s)| \leq c [|s|^2 + |s|^{p+1}],$$ for all $s\in \mathbb{R}$ and $k\in \mathbb{N}$.

Since $1\leq p < \frac{n}{n-2}$ if $n\geq 3$ and $ \frac{n}{n-2} <  \frac{n+2}{n-2}$ we obtain $2\leq p+1 < \frac{2n}{n-2}=2^{\ast}$. Consequently, there exists $\varepsilon >0$ such that $p+1+\varepsilon=2^{\ast}$.  Then, $H_0^1(\Omega) \hookrightarrow L^{p+1+\varepsilon}(\Omega)$ and, consequently,
\begin{eqnarray}\label{bound data Lp'}
\int_\Omega |F_k(u_{0,k})|^\frac{p+1+\varepsilon}{p+1}\,dx &\leq& c\int_\Omega |u_{0,k}|^{\frac{2(p+1+\varepsilon)}{p+1}} + |u_{0,k}|^{p+1+\varepsilon} \,dx\\
&\lesssim&  ||u_{0,k}||_{H_0^1(\Omega)}^{p+1+\varepsilon} \leq C.\nonumber
\end{eqnarray}
Analogously,
\begin{eqnarray}\label{bound on Fk'}
\int_\Omega |F_k(u_{k}(x,t_0))|^\frac{p+1+\varepsilon}{p+1}\,dx
\lesssim ||u_{k}(\cdot,t_0)||_{H_0^1(\Omega)}^{p+1+\varepsilon} \leq C E_{u_k}(0)^{p+1+\varepsilon},
\end{eqnarray}
for all $t_0\in [0,T]$.  The boundedness of $E_{u_k}(0)$ implies that there exists $\chi \in L^{\frac{2^{\ast}}{p+1}}(\Omega )$ verifying the following convergence:
\begin{eqnarray}\label{weak conv of Fk}
F_k(u_k(\cdot,t_0))\rightharpoonup \chi ~\hbox{ weakly in }L^{\frac{2^{\ast}}{p+1}}(\Omega ),~\hbox{ as } k \rightarrow +\infty.
\end{eqnarray}



In what follows we are going to prove that $\chi=F(u(\cdot,t_0))$. Indeed, from \eqref{Cauchy conv} we obtain $u_k(\cdot,t_0) \rightarrow u(\cdot,t_0)$ strongly in $L^2(\Omega)$.  Thus,
\begin{equation}\label{victor1}
u_k(x,t_0) \rightarrow u(x,t_0) \hbox{ a. e. in } \Omega.
\end{equation}

Note that,
\begin{eqnarray}\label{victor3}
&&|F_k(u_k(x,t_0))-F(u(x,t_0))| \\
&&\leq |F_k(u_k(x,t_0))-F(u_k(x,t_0))|+|F(u_k(x,t_0))-F(u(x,t_0))|.\nonumber
\end{eqnarray}

The convergence  (\ref{victor1}) and the continuity of $F$ imply
\begin{eqnarray}\label{victor4}
F(u_k(x,t_0)) \rightarrow F(u(x,t_0)) \hbox{ a. e. in } \Omega.
\end{eqnarray}

In light of inequality (\ref{victor3}), to prove that
\begin{eqnarray}\label{conv a.e. F_k}
F_k(u_k(x,t_0)) \rightarrow F(u(x,t_0)) \hbox{ a. e. in } \Omega.
\end{eqnarray}
it remains to prove that
\begin{eqnarray*}
F_k(u_k(x,t_0)) - F(u_k(x,t_0)) \rightarrow 0 \hbox{ a. e. in } \Omega,
\end{eqnarray*}

In fact, from (\ref{boundedness}), there exists a positive constant $L=L(x,t)>0$ such that \begin{eqnarray}\label{victor5}
|F_k(u_k(x,t_0))-F(u_k(x,t_0))| &=& \left|\int_{0}^{u_k(x,t_0)} f_k(s)ds-\int_{0}^{u_k(x,t_0)} f(s)ds\right|\nonumber\\
&\leq& \int_{-L}^{L} |f_k(s)-f(s)|ds =0,~\hbox{ if }k\geq L.
\end{eqnarray}

Therefore, combining \eqref{victor3}, \eqref{victor4} and \eqref{victor5}, we obtain (\ref{conv a.e. F_k}). Thus, from (\ref{bound on Fk'}) and Lions Lemma we deduce that
\begin{eqnarray}\label{main weak conv}
F_k(u_k(\cdot,t_0))\rightharpoonup F(u(\cdot,t_0)) ~\hbox{ weakly in }L^{\frac{2^{\ast}}{p+1}}(\Omega), \hbox{ as } k \rightarrow +\infty,
\end{eqnarray}
proving that $\chi=F(u(\cdot,t_0))$.

In addition, employing Strauss Lemma we also deduce that
\begin{eqnarray}\label{main strong conv}
F_k(u_k(\cdot, t_0))\rightarrow F(u(\cdot,t_0)) ~\hbox{ strongly in }L^r(\Omega),~\hbox{ as } k \rightarrow +\infty,
\end{eqnarray}
 for all $1\leq r < \frac{2^{\ast}}{p+1}$ and $t \in [0,T]$.

\medskip

Now we are in a position to establish the following result:
\begin{Theorem}\label{theo 2}
 Assume that $a\in L^\infty(\Omega)\cap C^0(\overline{\omega})$ is a nonnegative function and $f\in C^1(\mathbb{R})$ satisfies $f(s)s\geq 0$ for all $s\in \mathbb{R}$.  In addition, suppose that $f$ verifies assumption (\ref{ass on f}) with $1\leq p < \frac{n}{n-2},n\geq 3$ and $p\geq 1, n=1,2$ and assumption (\ref{ass on F}).  Then, given $\{u_0,u_1\}\in H_0^1(\Omega) \times L^2(\Omega)$ problem (\ref{eq:*}) has an unique global solution in the class
 {\small
 $$u\in C^0([0,T];H_0^1(\Omega)),~\partial_t u \in C^0([0,T]; L^2(\Omega)),~\partial_t^2 u \in L^2(0,T; H^{-1}(\Omega) ).$$}

In addition, the energy identity is verified
\begin{eqnarray}\label{energyidentity}\qquad
E_{u}(t):= \frac12\int_\Omega |\partial_t u(x,t)|^2 + |\nabla u(x,t)|^2\,dxdt + \int_\Omega F(u(x,t))\,dxdt.
\end{eqnarray}
 \end{Theorem}

\section{Exponential Decay to Problem (\ref{eq:*})}

Throughout this section we will assume that $1\leq p < \frac{n}{n-2}$ if $n\geq 3$ and $p\geq 1$ if $n=1,2$.  Under these conditions we have the following embeddings:
\begin{eqnarray}\label{tower}
H^1_0(\Omega) \hookrightarrow L^{2p}(\Omega)  \hookrightarrow L^{p}(\Omega).
\end{eqnarray}
Consider the auxiliary problem
\begin{equation}\label{Aux Prob}
\left\{
\begin{aligned}
&{\partial_t^2 u_k -\Delta u_k + f_k(u_k) - \operatorname{div}(a(x) \nabla \partial_t u_k + \frac{1}{k}\, b(x) \partial_t u_k= 0\quad \hbox{in}\,\,\,\Omega \times (0, +\infty),}\\\
&{u_k=0\quad \hbox{on}\quad \partial \Omega \times (0,+\infty ),}\\\
&{u_k(x,0)=u_{0,k}(x);\quad \partial_t u_k(x,0)=u_{1,k}(x),\quad x\in\Omega,}
\end{aligned}
\right.
\end{equation}
whose associated energy functional is given by
\begin{eqnarray}\label{energy}\qquad
E_{u_k}(t):= \frac12\int_\Omega |\partial_t u_k(x,t)|^2 + |\nabla u_k(x,t)|^2\,dxdt + \int_\Omega F_k(u_k(x,t))\,dxdt,
\end{eqnarray}
where  $F_k(\lambda) = \int_0^\lambda F_k(s) \,ds$ and the energy identity reads as follows
\begin{eqnarray}\label{ident energy mu}
E_{u_k}(t_2) - E_{u_k}(t_1) = - \int_{t_1}^{t_2}\int_\Omega  a(x) |\nabla \partial_t u_k|^2 +\frac{1}{k}\, b(x) |\partial_t u_k|^2 \,dxdt,
\end{eqnarray}
for all $0 \leq t_1 \leq t_2 <+\infty$.

Let $T_0>0$ be associated to the geometric control condition, that is, every ray of the geometric optics enters $\omega:=\Omega \backslash A$ in a time $T^*<T_0$. Thus, our goal is to prove the observability inequality established in the following lemma.

\begin{Lemma} ~ There exists $k_0 \geq 1$ such that for every $k\geq k_0$, the corresponding solution  $u_k$ of \eqref{Aux Prob} satisfies the inequality
\begin{eqnarray}\label{obs ineq}
E_{u_k}(0) \leq C \left(\int_{0}^{T}\int_\Omega  a(x) |\nabla \partial_t u_k|^2 \,dxdt +\frac{1}{k} \,\int_{0}^{T}\int_\Omega  b(x) |\partial_t u_k|^2\,dxdt\right),
\end{eqnarray}
for all $T>T_0$ and for some positive constant $C=C(||\{u_0,u_1\}||_{H_0^1(\Omega)\times L^2(\Omega)})$.
\end{Lemma}
\begin{proof}
The initial datum $\{u_0,u_1\}\in H_0^1(\Omega)\times L^2(\Omega)$ in the original problem (\ref{eq:*}) is either zero or not zero.

In the first case, when $\{u_0,u_1\} =(0,0)$ and, observing (\ref{conv init data}), we can consider  $\{u_{0,k},u_{1,k}\} =(0,0)$ for all $k\geq 1$ and the corresponding unique solution to the auxiliary problem (\ref{eq:AP}) will be $u_k \equiv 0$.  Then, (\ref{obs ineq}) is verified.

In the second case, there exists a positive number $R>0$ such that
$$0< ||\{u_0,u_1\}||_{H_0^1(\Omega)\times L^2(\Omega)} <R,$$
consider, for instance $R=2||\{u_0,u_1\}||_{H_0^1(\Omega)\times L^2(\Omega)}$.

Therefore, there exists, $k_0 \geq 1$ such that for all $k \geq k_0$, $\{u_{0,k},u_{1,k}\}$ satisfies
\begin{eqnarray}\label{sequencebound}
||\{u_{0,k},u_{1,k}\}||_{H_0^1(\Omega)\times L^2(\Omega)} <R.
\end{eqnarray}

We are going to prove that under condition (\ref{sequencebound}) on the initial datum, the corresponding solution $u_k$ to (\ref{eq:AP}) satisfies (\ref{obs ineq}).  Our proof relies on contradiction arguments.  So, if (\ref{obs ineq}) is false, then there exists $T>T_0$ such that for every $k\geq 1$ and every constant $C>0$, there exists an initial datum $\{u^C_{0,k},u^C_{1,k}\}$ verifying (\ref{sequencebound}), whose corresponding solution $u^C_k$ violates (\ref{obs ineq}).

In particular, for every $k\geq 1$ and $C=m\in\mathbb{N}$, we obtain the existence of an initial datum $\{u^m_{0,k},u^m_{1,k}\}$ verifying (\ref{sequencebound}) and whose corresponding solution $u^m_k$ satisfies

\begin{eqnarray}\label{false}
E_{u^m_k}(0)> m \left(\int_{0}^{T}\int_\Omega  a(x) |\nabla \partial_t u^m_k|^2 \,dxdt +\frac{1}{k} \,\int_{0}^{T}\int_\Omega  b(x) |\partial_t u^m_k|^2\,dxdt\right).
\end{eqnarray}

Then, we obtain a sequence $\{u_k^m\}_{m\in \mathbb{N}}$ of solutions to problem (\ref{eq:AP}) such that
\begin{eqnarray*}
\lim_{m \rightarrow +\infty} \frac{E_{u_k^m}(0)}{\int_{0}^{T}\int_\Omega \left( a(x) |\nabla \partial_t u_k^m|^2 + \frac{1}{k} \,b(x) |\partial_t u_k^m|^2 \right)\,dxdt}=+\infty.
\end{eqnarray*}

Equivalently
\begin{eqnarray}\label{normal conv}
\lim_{m \rightarrow +\infty}\frac{\int_{0}^{T}\int_\Omega \left( a(x) |\nabla \partial_t u_k^m|^2 +\frac{1}{k}\, b(x) |\partial_t u_k^m|^2 \right)\,dxdt}{E_{u_k^m}(0)}=0.
\end{eqnarray}

Since $E_{u_k^m}(0)$ is bounded,  (\ref{normal conv}) yields
\begin{eqnarray}\label{conv damp}
\lim_{m \rightarrow +\infty}\int_{0}^{T}\int_\Omega \left( a(x) |\nabla \partial_t u_k^m|^2 +\frac{1}{k}\, b(x) |\partial_t u_k^m|^2 \right)\,dxdt=0.
\end{eqnarray}

Furthermore, there exists a subsequence of $\{u_k^m\}_{m\in \mathbb{N}}$, Still denoted by $\{u_k^m\}$ , verifying the following convergences:
\begin{eqnarray}
&&u_k^m \rightharpoonup u_k \hbox{ weakly-star in } L^{\infty}(0,T; H_0^1(\Omega)),~\hbox{ as }m\rightarrow +\infty,\label{conv1'}\\
&&\partial_t u_k^m \rightharpoonup \partial_t u_k \hbox{ weakly-star in } L^{\infty}(0,T; L^2(\Omega)),~\hbox{ as }m\rightarrow +\infty,\label{conv2'}\\
&& u_k^m \rightarrow u_k \hbox{ strongly in } L^{\infty} (0,T; L^q(\Omega)), ~\hbox{ as }m\rightarrow +\infty, ~\hbox{ for all } q\in \left[2, \frac{2n}{n-2}\right),\label{conv3'}
\end{eqnarray}
where the last convergence is obtained using Aubin-Lions-Simon Theorem (see \cite{Simon}).  The proof is divided into two distinguished cases: $u_k\ne 0$ and $u_k=0$.

\medskip
Case~(a):~$u_k\ne 0$. ~
\medskip

For $m\in \mathbb{N}$, $u_k^m$ is the solution to the problem
\begin{equation*}
\left\{
\begin{aligned}
&{\partial_t^2 u_k^m -\Delta u_k^m + f_k(u_k^m) - \operatorname{div}(a(x) \nabla \partial_t u_k^m + \frac{1}{k}\, b(x) \partial_t u_k^m= 0\quad \hbox{in}\,\,\,\Omega \times (0, +\infty),}\\\
&{u_k^m=0\quad \hbox{on}\quad \partial \Omega \times (0,+\infty ),}\\\
&{u_k^m(x,0)=u_{0,k}^m(x);\quad \partial_t u_k^m(x,0)=u_{1,k}^m(x),\quad x\in\Omega.}
\end{aligned}
\right.
\end{equation*}

Taking (\ref{conv damp}) into consideration we obtain
\begin{equation}\label{limit1'}
\left\{
\begin{aligned}
&{\partial_t^2 u_k -\Delta u_k + f_k(u_k)= 0\quad \hbox{in}\,\,\,\Omega \times (0, +\infty),}\\\
&{u_k=0\quad \hbox{on}\quad \partial \Omega \times (0,+\infty ),}\\\
&{\partial_t u_k}=0 \hbox{ a.e. in }\Omega\backslash \mathcal{C}.
\end{aligned}
\right.
\end{equation}

Defining $y_k = \partial_t u_k$, the above problem yields
\begin{equation*}
\left\{
\begin{aligned}
&{\partial_t^2 y_k -\Delta y_k + f_k'(u_k)y_k= 0\quad \hbox{in}\,\,\,\Omega \times (0, +\infty),}\\\
&{y_k=0\quad \hbox{on}\quad \partial \Omega \times (0,+\infty ),}\\\
&{y_k}=0 \hbox{ a.e. in }\Omega\backslash \mathcal{C}.
\end{aligned}
\right.
\end{equation*}

Once $ f_k'(u_k)\in L^\infty (\Omega \times (0,T))$ since $f_k$ is globally Lipschitz, for each $k \in m\in \mathbb{N}$, we deduce from the uniqueness theorem due to Duyckaerts, Zhang and Zuazua, (see Theorem 2.2 in \cite{DZZ}), that $y_k= \partial_t u_k\equiv 0$.  Returning to (\ref{limit1'}) we conclude that $u_k\equiv 0$ as well and we obtain the desired contradiction.

\begin{Remark}
	In the case $\rho=\sqrt{\operatorname{det}g^{ij}}$, we can apply the Theorem 8.1 from \cite{Triggiani} combined with a density argument (since the potential $f_k'(u_k)$ is essentially bounded) to obtain the desired unique continuation property.
\end{Remark}

\medskip
Case~(b):~$u_k= 0$. ~
\medskip

Setting
\begin{eqnarray}\label{def v_k}
\alpha_m := \sqrt{E_{u_k^m}(0)}, ~ \hbox{ and }~v_k^m:= \frac{u_k^m}{\alpha_m},
\end{eqnarray}
in light of (\ref{normal conv}), we obtain
\begin{eqnarray}\label{damping conv}
\lim_{m \rightarrow +\infty} \int_{0}^{T}\int_\Omega \left( a(x) |\nabla \partial_t v_k^m|^2 + \frac{1}{k}\,b(x) |\partial_t v_k^m|^2 \right)\,dxdt=0.
\end{eqnarray}

According to (\ref{def v_k}), the sequence  $\{v_k^m\}_{m\in \mathbb{N}}$ is the solution to the following problem:
\begin{equation}\label{eq:NP}
\left\{
\begin{aligned}
&{\partial_t^2 v_k^m -\Delta v_k^m + \frac{1}{\alpha_m} f_k(u_k^m)  - \operatorname{div}(a(x) \nabla \partial_t v_k^m + \frac{1}{k} \,b(x) \partial_t v_k^m= 0\quad \hbox{in}\,\,\,\Omega \times (0, T),}\\\
&{v_k^m=0\quad \hbox{on}\quad \partial \Omega \times (0,T),}\\\
&{v_k^m(x,0)=\frac{u_{0,k}^m}{\alpha_m};\quad \partial_t v_k^m(x,0)=\frac{u_{1,k}^m}{\alpha_m}}
\end{aligned}
\right.
\end{equation}
and the associated energy functional is given by
\begin{eqnarray*}
E_{v_k^m}(t) = \frac12 \int_\Omega \left(|\partial_t v_k^m|^2 + |\nabla v_k^m|^2 \right)\,dx + \frac{1}{\alpha_m^2} \int_\Omega F_k(u_k^m)\,dx,
\end{eqnarray*}
since
$$
\frac{1}{\alpha_m}\int_\Omega f_k(u_k^m) \partial_t v_k^m\,dx =\frac{1}{\alpha_m^2} \frac{d}{dt}\int_\Omega F(u_k^m)\,dx.$$

Note that $E_{v_k^m}(t)= \frac{1}{\alpha_m^2} E_{u_k^m}(t)$ for all $t\geq 0$ and, in particular, for $t=0$
\begin{eqnarray}\label{norm initial energy}
E_{v_k^m}(0)= \frac{1}{\alpha_m^2} E_{u_k^m}(0)=1,~\hbox{ for all }m\in \mathbb{N}.
\end{eqnarray}

In order to achieve the contradiction we are going to prove that
\begin{eqnarray}\label{main goal}
\lim_{m\rightarrow +\infty} E_{v_k^m}(0)=0.
\end{eqnarray}

Indeed, initially, we observe that (\ref{norm initial energy}) yields the existence of a subsequence of $\{v_k^m\}_{m\in \mathbb{N}}$, reindexed again by $\{v_k^m\}$, such that
\begin{eqnarray}
&&v_k^m \rightharpoonup v_k \hbox{ weakly-star in } L^{\infty}(0,T; H_0^1(\Omega)),~\hbox{ as }m\rightarrow +\infty,\label{conv1}\\
&&\partial_t v_k^m \rightharpoonup \partial_t v_k \hbox{ weakly-star in } L^{\infty}(0,T; L^2(\Omega)),~\hbox{ as }m\rightarrow +\infty,\label{conv2}\\
&& v_k^m \rightarrow v_k \hbox{ strongly in } L^{\infty} (0,T; L^q(\Omega)), ~\hbox{ as }m\rightarrow +\infty, ~\hbox{ for all } q\in \left[2, \frac{2n}{n-2}\right).\label{conv3}
\end{eqnarray}

For some eventual subsequence, we have that $\alpha_m \rightarrow \alpha$ with $\alpha\geq0$.

\medskip

If $\alpha>0$, let $A_\varepsilon:= A + \overline{B_{\varepsilon/2}(0)}$ and $\mathcal{C} :=A\backslash A_\varepsilon=\{x\in A: d(x,y)>\varepsilon/2,~y\in \partial A\}$.

Passing to the limit in (\ref{eq:NP}) and considering convergences (\ref{damping conv}),(\ref{conv1}) - (\ref{conv3}), we deduce
\begin{equation}\label{limit1}
\left\{
\begin{aligned}
&{\partial_t^2 v_k -\Delta v_{k} + \frac{1}{\alpha} f_k(u_k) = 0\quad \hbox{ in }\,\,\,\Omega \times (0, T),}\\\
&{v_k=0\quad \hbox{on}\quad \partial \Omega \times (0,T),}\\\
&{\partial_t v_k=0 \hbox{ a.e. in }\Omega \backslash \mathcal{C}.}
\end{aligned}
\right.
\end{equation}

The above problem yields, for $w_k= \partial_t v_k$, in the distributional sense,
\begin{equation}
\left\{
\begin{aligned}
&{\partial_t^2 w_k -\Delta w_k + \frac{1}{\alpha} f_k'(u_k) w_k = 0\quad \hbox{ in }\,\,\,\Omega \times (0, T),}\\\
&{w_k=0\quad \hbox{on}\quad \partial \Omega \times (0, T),}\\\
&{w_k=0 \hbox{ a.e. in }\Omega \backslash \mathcal{C}.}
\end{aligned}
\right.
\end{equation}

Once $\frac{1}{\alpha} f_k'(u_k)\in L^\infty (\Omega \times (0,T))$, using again standard uniqueness theorem, we conclude that $w_k= \partial_t v_k\equiv 0$, and, therefore, returning to (\ref{limit1}) we deduce that $v_k\equiv 0$.

\medskip

If $\alpha=0$, first, observe that hypothesis (\ref{Lipschitz}) yields
\begin{eqnarray*}
\frac{1}{\alpha_m^2} |f_k(u_k^m)|^2 \leq c \frac{1}{\alpha_m^2} |u_k^m|^2 =c \frac{1}{\alpha_m^2}\alpha_m^2 |v_k^m|^2,
\end{eqnarray*}
and
\begin{eqnarray}\label{crucial bound'}
\frac{1}{\alpha_m^2} \int_0^T \int_\Omega |f_k(u_k^m)|^2 \,dxdt \leq  c \int_0^T\int_\Omega |v_k^m|^2\,dxdt.
\end{eqnarray}

We are going to prove that
\begin{equation}\label{nl1}
\frac{1}{\alpha_m}f_k(\alpha_m v_k^m) \rightharpoonup f'(0)v_k \hbox{ in } L^2(0,T;L^2(\Omega)) \hbox{ as } m \rightarrow \infty.
\end{equation}

Since
\begin{equation*}
\frac{1}{\alpha_m}f_k(\alpha_m v_k^m) - f'(0)v_k=\frac{1}{\alpha_m}f_k(\alpha_m v_k^m)-\frac{1}{\alpha_m}f(\alpha_m v_k^m)+\frac{1}{\alpha_m}f(\alpha_m v_k^m) - f'(0)v_k,
\end{equation*}
if we prove that
\begin{equation}\label{nl2}
\frac{1}{\alpha_m}f_k(\alpha_m v_k^m)-\frac{1}{\alpha_m}f(\alpha_m v_k^m) \rightarrow 0 \hbox{ in } L^2(0,T;L^2(\Omega))
\end{equation}
and
\begin{equation}\label{nl3}
\frac{1}{\alpha_m}f(\alpha_m v_k^m) - f'(0)v_k \rightharpoonup 0
\hbox{ in } L^2(0,T;L^2(\Omega)),
\end{equation}
as $m \rightarrow \infty$, we prove (\ref{nl1}).

To prove (\ref{nl2}), let's consider $$\Omega_m^t=\{x \in \Omega: |u_k^m(x,t)|>k\}.$$

Employing definition (\ref{trunc func}), $|f_k(\alpha_m v_k^m)- f(\alpha_m v_k^m)|=0$ in $\Omega \setminus \Omega_m^t$.  Then, hypotheses (\ref{main ass on f}) and (\ref{ass on f}) yield
\begin{small}
\begin{align*}
&\left\| \frac{1}{\alpha_m}f_k(\alpha_m v_k^m)-\frac{1}{\alpha_m}f(\alpha_m v_k^m) \right\|_{L^2(0,T;L^2(\Omega))}^2 \\
= {}& \int_{0}^{T}\int_{\Omega_m^t} \left|\frac{1}{\alpha_m}f_k(\alpha_m v_k^m)-\frac{1}{\alpha_m}f(\alpha_m v_k^m) \right|^2 \, dxdt\\
= {}&\frac{1}{\alpha_m^2}\int_{0}^{T}\int_{\Omega_m^t} \left|f_k(\alpha_m v_k^m)-f(\alpha_m v_k^m) \right|^2 \, dxdt\\
\lesssim {}& \frac{1}{\alpha_m^2}\int_{0}^{T}\int_{\Omega_m^t}|f_k(\alpha_m v_k^m )|^2\, dx dt+\frac{1}{\alpha_m^2}\int_{0}^{T}\int_{\Omega_m^t} |f(\alpha_m v_k^m)|^2 \, dx dt \\
\lesssim {}& \int_{0}^{T}\int_{\Omega_m^t} |f(k)|^2+|f(-k)|^2 \, dx dt+\frac{1}{\alpha_m^2}\int_{0}^{T}\int_{\Omega_m^t} |\alpha_m v_k^m|^2 + |\alpha_m v_k^m|^{2p} \, dx dt\\
\lesssim {}& \int_{0}^{T}\int_{\Omega_m^t} |k|^2+|k|^{2p} \, dx dt+\frac{1}{\alpha_m^2}\int_{0}^{T}\int_{\Omega_m^t} |\alpha_m v_k^m|^2 + |\alpha_m v_k^m|^{2p} \, dx dt\\
\lesssim {}& \frac{1}{\alpha_m^2}\int_{0}^{T}\int_{\Omega_m^t} |u_k^m|^2 + |u_k^m|^{2p} \, dx dt+\frac{1}{\alpha_m^2}\int_{0}^{T}\int_{\Omega_m^t} |\alpha_m v_k^m|^2 + |\alpha_m v_k^m|^{2p} \, dx dt.
\end{align*}
\end{small}

Since $p>1$, $k\geq 1$ and $ k <|u_k^m|=|\alpha_m v_k^m|$ in $\Omega_m^t$, we obtain
\begin{small}
\begin{align*}
\left\| \frac{1}{\alpha_m}f_k(\alpha_m v_k^m)-\frac{1}{\alpha_m}f(\alpha_m v_k^m) \right\|_{L^2(0,T;L^2(\Omega))}^2
\lesssim {}& \frac{1}{\alpha_m^2}\int_{0}^{T}\int_{\Omega_m^t} |\alpha_m v_k^m|^{2p} \, dx dt\\
\lesssim {}& \alpha_m^{2(p-1)} \|v_k^m\|_{L^{2p}(0,T;L^{2p}(\Omega))}^{2p}\rightarrow 0, \hbox{ as } m \rightarrow \infty,
\end{align*}
\end{small}
which proves the convergence \eqref{nl2}.

On the other hand, $f\in C^2(\mathbb{R})$ and, consequently, from Taylor's Theorem and (\ref{main ass on f}) we have
\begin{eqnarray}\label{Taylor}
f(s) = f'(0) s + R(s), \hbox{ where } |R(s)|\leq C(|s|^2 + |s|^p).
\end{eqnarray}

Hence
\begin{eqnarray}\label{ident}
&&\frac{1}{\alpha_m}f(\alpha_m v_k^m)= f'(0) v_k^m + \frac{R(\alpha_m v_k^m)}{\alpha_m}\label{T1}
\end{eqnarray}
and
\begin{equation}\label{T2}
 \left|\frac{R(\alpha_m v_k^m)}{\alpha_m}\right| \leq C \left(\alpha_m|v_k^m|^2 + |\alpha_m|^{p-1} |v_k^m|^p\right).
\end{equation}

In light of identity (\ref{Taylor}), we establish $ \frac{R(\alpha_m v_k^m)}{\alpha_m}=\frac{f(\alpha_m v_k^m)}{\alpha_m}-f'(0)v_k^m$ and hypotheses (\ref{main ass on f}) and (\ref{ass on f}) imply that  $|f(\alpha_m v_k^m)| \lesssim |\alpha_m v_k^m|+|\alpha_m v_k^m|^p$.  Then, we deduce that
\begin{equation*}
\left\| \frac{R(\alpha_m v_k^m)}{\alpha_m} \right\|_{L^2(0,T;L^2(\Omega))}^2 \lesssim \|v_k^m\|_{L^2(0,T;L^2(\Omega))}^2+|\alpha_m|^{2(p-1)}\|v_k^m\|_{L^{2p}(0,T;L^{2p}(\Omega))}^{2p} \leq C,
\end{equation*}
for some constant $C>0$.   We obtain a subsequence of $\frac{R(\alpha_m v_k^m)}{\alpha_m}$ and $\gamma \in L^2(0,T;L^2(\Omega))$ such that
\begin{equation}\label{resto1}
\frac{R(\alpha_m v_k^m)}{\alpha_m} \rightharpoonup \gamma \hbox{ in } L^2(0,T;L^2(\Omega)).
\end{equation}

Besides, employing inequality (\ref{T2}) and observing (\ref{tower}), we get
\begin{eqnarray}\label{resto2}
\left|\left| \frac{R(\alpha_m v_k^m)}{\alpha_m}\right|\right|_{L^1(0,T;L^1(\Omega))}&\lesssim& \int_0^T \int_\Omega \alpha_m |v_k^m|^2 \,dxdt + \int_0^T\int_\Omega \alpha_m^{p-1} |v_k^m|^{p}\,dxdt \nonumber\\
&=& \alpha_m\int_{0}^{T}\|v_k^m\|_{L^2(\Omega)}^2 \, dt+\alpha_m^{p-1}\int_{0}^{T}\|v_k^m\|_{L^p(\Omega)}^{p}\, dt \nonumber \\
&=&  \alpha_m ||v_k^m||_{L^2(0,T; L^2(\Omega))}^2 + \alpha_m^{p-1}||v_k^m||_{L^p(0,T;L^{p}(\Omega))}^{p}\rightarrow 0.
\end{eqnarray}

From \eqref{resto1} and \eqref{resto2} we conclude that
\begin{equation}\label{resto3}
\frac{R(\alpha_m v_k^m)}{\alpha_m} \rightharpoonup 0 \hbox{ in } L^2(0,T;L^2(\Omega)).
\end{equation}

Observing \eqref{conv3}, \eqref{ident} and \eqref{resto3}, the convergence \eqref{nl3} is proved.

\begin{Remark}
	The case $p=1$ is trivially contemplated once the truncation is not necessary.
\end{Remark}

Since convergences (\ref{nl2}) and (\ref{nl3}) are proved, we conclude convergence (\ref{nl1}).

Passing to the limit in (\ref{eq:NP}) as $m\rightarrow +\infty$, we obtain
\begin{equation}\label{limit2}
\left\{
\begin{aligned}
&{\partial_t^2 v_k -\Delta v_k + f'(0) v_k  = 0\quad \hbox{ in }\,\,\,\Omega \times (0, T),}\\\
&{v_k=0\quad \hbox{on}\quad \partial \Omega \times (0, T),}\\\
&{\partial_t v_k=0 \hbox{ a.e. in }\Omega \backslash  \mathcal{C},}
\end{aligned}
\right.
\end{equation}
and defining $w_k= \partial_t v_k$, it satisfies the following problem:
\begin{equation}
\left\{
\begin{aligned}
&{\partial_t^2 w_k -\Delta w_k + f'(0) w_k = 0\quad \hbox{ in }\,\,\,\Omega \times (0, T),}\\\
&{w_k=0\quad \hbox{on}\quad \partial \Omega \times (0, T),}\\\
&{w_k=0 \hbox{ a.e. in }\Omega \backslash  \mathcal{C}.}
\end{aligned}
\right.
\end{equation}

Using standard uniqueness results we obtain that $w_k= \partial_t v_k\equiv 0$ and returning to \eqref{limit2} we deduce that $v_k\equiv 0$.

Then, in both cases $\alpha =0$ and $\alpha >0$, we obtain that $v_k\equiv 0$.  Consequently,  inequality (\ref{crucial bound'}) and convergence (\ref{conv3}) yield that
\begin{eqnarray}\label{crucial bound}
\frac{1}{\alpha_m^2} \int_0^T \int_\Omega |f_k(u_k^m)|^2 \,dxdt \rightarrow 0 \hbox{ in } L^2(0,T,L^2(\Omega)).
\end{eqnarray}

Now, let's consider $\Box= \partial_t - \Delta$ the d'Alembert operator. Thanks to (\ref{eq:NP}), we have that
\begin{eqnarray*}
\Box v_k^m = -\frac{1}{\alpha_m} f_k(u_k^m)  + \operatorname{div}(a(x) \nabla \partial_t v_k^m - \frac{1}{k}\,b(x) \partial_t v_k^m,
\end{eqnarray*}
that is,
\begin{eqnarray}\label{main identity}
\partial_t\Box v^m =\partial_t \left(  -\frac{1}{\alpha_m} f_k(u_k^m)  + \operatorname{div}(a(x) \nabla \partial_t v_k^m) - \frac{1}{k}\,b(x) \partial_t v_k^m\right).
\end{eqnarray}

Let us analyse the terms on the RHS of (\ref{main identity}).
\smallskip

\noindent{Analysis of $I_1^m:=  -\frac{1}{\alpha_m} f_k(u_k^m) $.}

\smallskip

Employing convergence (\ref{crucial bound}), we deduce that
$ I_1^m\rightarrow 0$ strongly in $ L^2(\Omega \times (0,T))$ as $m \rightarrow +\infty$.

\medskip

\noindent{Analysis of $I_3^m:= - \frac{1}{k}\,b(x) \partial_t v_k^m$.}

\medskip

We trivially obtain from (\ref{damping conv}) that $- \frac{1}{k}\,b(x) \partial_t v_k^m \rightarrow 0$ strongly in $L^2(\Omega \times (0,T))$  as $m \rightarrow +\infty$.

\medskip

\noindent{Analysis of $I_2^m:= \operatorname{div}(a(x) \nabla \partial_t v_k^m)$.}

\medskip

Recalling convergence (\ref{damping conv}), we deduce that $a(x) \nabla \partial_t v_k^m \rightarrow 0 \hbox{ strongly in } L^2(\Omega \times (0,T))$ as $m \rightarrow +\infty$ and, consequently, $\operatorname{div}(a(x) \nabla \partial_t v_k^m) \rightarrow 0 \hbox{ strongly in } H^{-1}_{loc}(\Omega \times (0,T))$ as $m \rightarrow +\infty$.

The above convergences yield
\begin{eqnarray*}\quad
\partial_t \left(  -\frac{1}{\alpha_m} f_k(u_k^m) + \operatorname{div}(a(x) \nabla \partial_t v_k^m) - \frac{1}{k}\,b(x) \partial_t v_k^m\right) \rightarrow 0 \hbox{ strongly in } H^{-2}_{loc}(\Omega \times (0,T)),
\end{eqnarray*}
that is, from (\ref{main identity}) we conclude that
\begin{eqnarray}\label{Gerard1}
\Box \partial_t v_k^m \rightarrow 0 \hbox{ strongly in } H^{-2}_{loc}(\Omega \times (0,T)),
\end{eqnarray}
as $m \rightarrow +\infty$.

Let $\mu$ be the microlocal defect measure associated with $\{\partial_t v_k^m\}$ in $L^2_{loc}(\Omega \times (0,T))$.  The convergence (\ref{Gerard1}) guarantees that the $\hbox{supp}(\mu)$ is contained in the characteristic set of the wave operator $\{\tau^2=||\xi||^2\}$~(see G\'erard  \cite{Gerad} (Proposition 2.1 and Corollary 2.2) or Theorem \ref{Theoremak 4.55} in the Appendix).

Our goal is to propagate the convergence of $\partial_t v_k^m$ from $L^2((\Omega \backslash \mathcal{C}) \times (0,T))$ to the whole space $L^2(\Omega \times (0,T))$.  However, the convergence in (\ref{Gerard1}) is not regular enough to guarantee the desired propagation.  This lack of regularity comes from the following convergence:
\begin{eqnarray*}
\partial_t \left(\operatorname{div}(a(x) \nabla \partial_t u_k^m) \right)  \rightarrow 0 \hbox{ strongly in } H^{-2}_{loc}(\Omega \times (0,T)), \hbox{ as }m \rightarrow +\infty,
\end{eqnarray*}

At this point, the presence of the frictional damping in the neighbourhood $V_{\varepsilon}$  of the boundary $\partial A$ of $A:=\{x\in \Omega: a(x)=0\}$ plays an essential role to obtain the propagation to the whole space.  Indeed, note that $a(x) \nabla \partial_t u_k^m=0$ in $A \times (0,T)$, consequently,  from (\ref{main identity})  we have that
\begin{eqnarray}\label{Gerard2}
\Box \partial_t v_k^m = \partial_t \left(  -\frac{1}{\alpha_m} f_k(u_k^m) - \frac{1}{k}\,b(x) \partial_t v_k^m\right) \rightarrow 0 \hbox{ in } H^{-1}_{loc}(\hbox{int}\, A \times (0,T)),
\end{eqnarray}
as $m \rightarrow +\infty.$

The above convergence yields that $\mu$ propagates along the bicharacteristic flow of the D'Alembertian operator, which means that if there is $\omega_0=(t_0,x_0,\tau_0,\xi_0)$ such that $\omega_0=(t_0,x_0,\tau_0,\xi_0) \notin \hbox{supp}(\mu)$, the whole bicharacteristic issued from $\omega_0$ does not belong to $\hbox{supp}(\mu)$ (see Proposition \ref{proposition5.1} and Theorem \ref{Theoremak 4.60} in the Appendix). However, since $\hbox{supp}(\mu) \subset \mathcal{C} \times (0,T)\subset A \times (0,T)$, and the frictional damping acts in both sides of the boundary $\partial A$, we can propagate the kinetic energy from $(V_{ {\varepsilon} /2} \cap A) \times (0,T)$ towards the set $\mathcal{C} \times (0,T)$.

On the other hand, Poincar\'e inequality combined with the existence of the dissipative effects, yield
\begin{eqnarray}\label{kinectic convergence Omega-C}
\partial_t  v_k^m \rightarrow 0 \hbox{ in } (\Omega\backslash \mathcal{C}) \times (0,T),
\end{eqnarray}
consequently, we obtain
\begin{eqnarray}\label{kinectic convergence}
\int_0^T \int_\Omega |\partial_t  v_k^m(x,t)|^2\,dxdt \rightarrow 0, ~\hbox{ as }m \rightarrow + \infty.
\end{eqnarray}

Then, we prove that $ E_{v_k^m}(0)\rightarrow 0$ as $m \rightarrow +\infty$.  In fact, consider the following cut-off function
\begin{align*}
&\theta\in C^{\infty}(0,T),  \quad    0\leq \theta(t) \leq 1,  \quad   \theta(t)=1 \ \mbox{in} \ (\varepsilon,T-\varepsilon), \varepsilon>0.
\end{align*}

Multiplying equation (\ref{eq:NP}) by $v_k^m \theta$ and integrating by parts, we infer
\begin{eqnarray}\label{equipartition}
&& -\int_0^T \theta(t)\int_\Omega |\partial_t v_k^m|^2\,dxdt - \int_0^T \theta'(t)\int_\Omega \partial_t v_k^m v_k^m\,dxdt\\
&&+\int_0^T \theta(t) \int_\Omega  |\nabla v_k^m|^2 \,dxdt + \frac{1}{\alpha_m}\int_0^T \theta(t) \int_\Omega f_k(v_k^m) v_k^m \,dxdt\nonumber\\
&&+ \int_0^T \theta(t) \int_\Omega a(x) \nabla \partial_t v_k^m \cdot \nabla v_k^m\,dxdt\nonumber\\
&&+\frac{1}{k}\,\int_0^T \theta(t) \int_\Omega b(x) \partial_t v_k^m  v_k^m\,dxdt =0.\nonumber
\end{eqnarray}

Considering the convergences (\ref{damping conv}), (\ref{conv1})-(\ref{conv3}) and (\ref{kinectic convergence}) and employing the fact that $v_k=0$,  identity (\ref{equipartition}) yields
\begin{equation}\label{convgradto0}
\lim_{m \rightarrow +\infty}\int_\varepsilon^{T-\varepsilon} \int_\Omega  |\nabla v_k^m|^2 +  \frac{1}{\alpha_m} f_k(v_k^m)v_k^m \,dxdt = 0.
\end{equation}

In addition,using the definition of functions $f_k$ and $F_k$ and taking advantage of property (\ref{ass on F}), the last convergence implies
\begin{eqnarray}\label{convFto0}
\lim_{m \rightarrow +\infty}\frac{1}{\alpha_m^2}\int_\varepsilon^{T-\varepsilon} \int_\Omega F_k(v_k^m) \,dxdt=0.
\end{eqnarray}

Convergences (\ref{kinectic convergence}), (\ref{convgradto0}) and (\ref{convFto0}) establish that
 $\int_\varepsilon^{T-\varepsilon} E_{v_k^m}(t) \rightarrow 0$ and recalling that the energy functional is a non-increasing function, we obtain
\begin{equation}\label{convenergyto0}
(T- 2\varepsilon) E_{v_k^m}(T-  \varepsilon) \rightarrow 0,~\hbox{as}~m\rightarrow +\infty.
\end{equation}

Using the energy identity we deduce that
\begin{equation*}
E_{v_k^m}(T- \varepsilon) - E_{v_k^m}(\varepsilon) = - \int_{\varepsilon}^{T- \varepsilon} \int_{\Omega} a(x)|\nabla \partial_t v_k^m|^2+ \frac{1}{k}\,b(x)|\partial_t v_k^m|^2 \,dxdt,
\end{equation*}
which, in view of  convergences (\ref{damping conv}) and (\ref{convenergyto0}) and the arbitrariness of $\varepsilon >0$, implies that $E_{v_k^m}(0) \rightarrow 0$ as $m\rightarrow +\infty$.

Then, according to (\ref{main goal}), the desired contradiction is achieved and we finish the proof.
\end{proof}

\medskip

In what follows, we are going to conclude the exponential stability to the problem (\ref{eq:*}).

Thanks to inequality (\ref{obs ineq}), the auxiliary problem (\ref{eq:AP}) satisfies the following observability inequality:
\begin{equation}\label{trunc observ ineq}
E_{u_k}(0) \leq C\,\int_0^{T}\int_{\Omega}\left(\frac{1}{k}\,b(x)|\partial_t u_k|^2 + a(x) |\nabla \partial_t u_k|^2 \right)\,dx\,dt,\hbox{ for all } T\geq T_0, \hbox{ and } k\in \mathbb{N},~ k \geq k_0,
\end{equation}
where $C$ is a positive constant which does not depend on $k\in \mathbb{\mathbb{N}}$.

Passing to the limit as $k\rightarrow +\infty$ and observing convergences (\ref{eq:3.54}), (\ref{Cauchy conv}), (\ref{damping conv'}) and (\ref{main strong conv}), the above inequality yields  the observability inequality associated to the original problem (\ref{eq:*}), that is,
\begin{equation}
\label{final obs_inequality}
E_{u}(0) \leq C\,\int_0^{T}\int_{\Omega} a(x) |\nabla \partial_t u|^2 \,dx\,dt, \hbox{ for all } T\geq T_0.
\end{equation}	 	 	

On the other hand, passing to the limit as $k\rightarrow +\infty$ and considering the same convergences (\ref{eq:3.54}), (\ref{Cauchy conv}), (\ref{damping conv'}) and (\ref{main strong conv}), identity (\ref{est2}) yields the identity associated to the original problem (\ref{eq:*}), namely,
\begin{eqnarray}\label{final energy ident}\quad
E_u(t_2) - E_u (t_1) + \int_{t_1}^{t_2} \int_\Omega  a(x) |\nabla \partial_t u|^2 \,dx\,dt=0,~\hbox{ for all }0\leq t_1 < t_2<+\infty.
\end{eqnarray}

Gathering together (\ref{final obs_inequality}), (\ref{final energy ident}), and since the map $t \mapsto E_{u}(t)$ is a non-increasing function, we obtain
\begin{equation}
\begin{aligned}
E_{u}(T_0)&\leq C\,\int_0^{T_0}\int_{\Omega}\left(a(x) |\nabla \partial_t u|^2\right)\,dx\,dt\\&=C\,\left(E_{u}(0)-E_{u}(T_0)\right),
\end{aligned}
\end{equation}
that is,
\begin{eqnarray}
\label{inequality}E_{u}(T_0)&\leq&\left(\frac{C}{1+C}\right)\,E_{u}(0).
\end{eqnarray}

Repeating the same steps for $m T_0$, $m\in \mathbb{N}, m\geq 1$, we deduce
$$E_{u}(m T_0)\leq\frac{1}{(1+\hat{C})^m}E_{u}(0),$$
where $\hat{C}=C^{-1}$. Consider $t\geq T_0$ and $t=m T_0+r,$ $0\leq r<T_0$. Thus,
$$E_{u}(t)\leq E_{u}(t-r)=E_{u}(m T_0)\leq \frac{1}{(1+\hat{C})^m}\,E_{u}(0)=\frac{1}{(1+\hat{C})^{\frac{t-r}{T_0}}}E_{u}(0).$$
	 	
Defining $\displaystyle C:=\textrm{e}^{\frac{r}{T_0}\ln(1+\hat{C})}$ and $\lambda_0:=\frac{\ln(1+\hat{C})}{T_0}>0,$ we obtain
\begin{equation}\label{exp'}
E_{u}(t)\leq C \,\textrm{e}^{- \lambda_0 t}E_{u}(0) \hbox{ for all } t\geq
T_0,
\end{equation}
which proves the exponential decay  to problem (\ref{eq:*}) and we prove the following result.

\begin{Theorem}\label{theo 3}
Under the assumptions of Theorem \ref{theo 2}, there exist positive constants $C$ and $\gamma$ such that the following exponential decay holds
\begin{equation}\label{exp}
	 	 	E_{u}(t)\leq C \,\textrm{e}^{- \lambda_0 t}E_{u}(0), \hbox{ for all } t\geq T_0.
\end{equation}
for every solution to problem (\ref{eq:*}), provided that the initial data are taken in bounded sets of the phase-space $\mathcal{H}:= H_0^1(\Omega) \times L^2(\Omega)$.
\end{Theorem}

\section{Appendix}

\subsection{The proof of convergence (\ref{eq:3.56}).}

In this section we prove the convergence (\ref{eq:3.56}) and remark the importance of the continuity of the Kelvin-Voigt damping function $a=a(x)$ in a neighbourhood of $\Omega \backslash A$.

\begin{Lemma}\label{aux lemma}
Let $\{f_k\}_{k\in \mathbb{K}}$ be a sequence of functions in $L^2(\mathcal{O})$, where $\mathcal{O}$ is open set of $\mathbb{R}^d$. Let us assume that $f_k \rightharpoonup f$ weakly in $L^2(\mathcal{O})$. Then, $f_k|_{\mathcal{O}\backslash K} \rightharpoonup  f|_{\mathcal{O}\backslash K}$ weakly in $L^2(\Omega\backslash K)$ for all compact set $K$ contained in $\mathcal{O}$.
\end{Lemma}

\begin{proof}
The proof consists in considering $\phi\in L^2(\mathcal{O}\backslash K)$ and defining
\begin{equation}
\Phi_K(x) :=\left\{
\begin{aligned}
&\phi~\hbox{ in } \mathcal{O}\backslash K,\\
& 0~\hbox{ in } K,
\end{aligned}
\right.
\end{equation}
which is a $L^2(\mathcal{O})$ function.
\end{proof}

\medskip

Let's prove convergence (\ref{eq:3.56}). For the sake of simplicity, we are going to consider $a\in C^{\infty}(\Omega \backslash A)$.  But, using distributions of finite order we should just consider $a$ of class $C^0$.

The boundedness of $\sqrt{a(x)} \nabla \partial_t u_k$ in $L^2(\Omega \times (0,T))$ implies that there exists $\chi \in L^2(\Omega \times (0,T))$ verifying
\begin{eqnarray}\label{convI}
\sqrt a(x)  \nabla \partial_t u_k \rightharpoonup \chi ~\hbox{ weakly in } L^2(\Omega \times (0,T)).
\end{eqnarray}

Then, since $a \equiv 0$ in $A$ it follows that $\chi := 0$ a. e. in $A$. It remains to prove that $\chi = \sqrt{a(x)} \nabla \partial_t u$ in $\Omega \backslash A$ and $a(x)>0$ in $\Omega \backslash A$.

Employing convergence (\ref{conver2}), that is, $\partial_t u_k \rightharpoonup \partial_t u$ weakly in $L^2(\Omega \times (0,T))$ and Lemma \ref{aux lemma}, we infer that $\partial_t u_k \rightharpoonup \partial_t u$ weakly in $L^2((\Omega\backslash A) \times (0,T))$ and, consequently,
$$\nabla \partial_t u_k \rightarrow \nabla \partial_t u \hbox{  in } \mathcal{D}'((\Omega\backslash A) \times (0,T)).$$

Since $a\in C^{\infty}(\Omega \backslash A)$ and $a(x)>0$ in $\Omega\backslash A$,
$$\sqrt{a} \nabla \partial_t u_k \rightarrow \sqrt{a} \nabla \partial_t u \hbox{  in } \mathcal{D}'((\Omega\backslash A) \times (0,T)).$$

In view of Lemma \ref{aux lemma}, we obtain that $\sqrt{a} \nabla \partial_t u_k \rightarrow \chi$ weakly in $L^2((\Omega\backslash A) \times (0,T))$ and we conclude that $\chi= \sqrt{a} \nabla \partial_t u$.

\subsection{Microlocal Analysis Background}

For the readers comprehension, we will announce some results which can be found in Burq and G\'erard \cite{Burq-Gerard}  and in G\'erard \cite{Gerad} and were used in the proof of the exponential stabilization.

\begin{Theorem}\label{Theoremak 4.49}
Let $\{u_n\}_{n\in \mathbb{N}}$ be  a bounded sequence in $L_{loc}^2(\mathcal{O})$ such that it converges weakly to zero in $L_{loc}^2(\mathcal{O})$. Then, there exists a subsequence $\{u_{\varphi(n)}\}$ and a positive Radon measure $\mu$ on $T^1\mathcal{O}:=\mathcal{O}\times S^{d-1}$ such that for all pseudo-differential operator $A$ of order $0$ on $\Omega$ which admits a principal symbol $\sigma_0(A)$ and for all $\chi \in C_0^\infty(\mathcal{O})$ such that $\chi \sigma_0(A)=\sigma_0(A)$, one has
\begin{eqnarray}\label{M6}
\left(A(\chi u_{\varphi(n)}), \chi u_{\varphi_n}\right)_{L^2}\underset{n\rightarrow +\infty}\longrightarrow \int_{\mathcal{O} \times S^{n-1}}\sigma_0(A)(x,\xi)\,d\mu(x,\xi).
\end{eqnarray}
\end{Theorem}

\begin{Definition}
${}$

\smallskip

Under the circumstances of Theorem \ref{Theoremak 4.49} $\mu$ is called the \underline{\bf microlocal defect measure} of the sequence $\{u_{\varphi(n)}\}_{n\in \mathbb{N}}$.
\end{Definition}

\begin{Remark}
Theorem \ref{Theoremak 4.49} assures that for all bounded sequence $\{u_n\}_{n\in \mathbb{N}}$ of $L_{loc}^2(\mathcal{O})$ which converges weakly to zero, the existence of a subsequence admitting a microlocal defect measure. We observe that from (\ref{M6}) in the particular case when $A=f\in C_0^\infty(\mathcal{O})$, it follows that
\begin{eqnarray}\label{M22}
\int_\Omega f(x) |u_{\varphi(n)}(x)|^2\,dx \rightarrow \int_{\mathcal{O} \times S^{d-1}}f(x)\,d\mu(x,\xi),
\end{eqnarray}
so that $u_{\varphi(n)}$ converges to $0$ strongly  if and only if $\mu=0$.
\end{Remark}

The second important result reads as follows.

\begin{Theorem}\label{Theoremak 4.55}
Let $P$ be a differential operator of order $m$ on $\Omega$ and let $\{u_n\}$ a bounded sequence of $L_{loc}^2(\mathcal{O})$ which converges weakly to $0$ and admits a m.d.m. $\mu$. The following statement are equivalents:
\begin{eqnarray*}
&(i)&Pu_n \underset{n \rightarrow +\infty}\longrightarrow 0 \hbox{ strongly in }H_{loc}^{-m}(\mathcal{O})~(m>0). \\
&(ii)& \hbox{supp} (\mu) \subset \{(x,\xi)\in \mathcal{O} \times S^{n-1}: \sigma_m(P)(x,\xi)=0\}.
\end{eqnarray*}
\end{Theorem}

\begin{Theorem}\label{Theoremak 4.56}
Let $P$ be a differential operator of order $m$ on $\Omega$, verifying $P^\ast=P$, and let $\{u_n\}$ be a bounded sequence in $L_{loc}^2(\mathcal{O})$ which converges weakly to $0$ and it admits a m.d.m. $\mu$. Let us assume that $P u_n \underset{n\rightarrow +\infty}\longrightarrow 0$ strongly in $H_{loc}^{1-m}(\mathcal{O})$. Then, for all function $a\in C^\infty(\mathcal{O} \times \mathbb{R}^n \backslash\{0\} )$ homogeneous of degree $1-m$ in the second variable and with compact support in the first one,
\begin{eqnarray}\label{M26}
\int_{\Omega \times s^{n-1}}\{a,p\}(x,\xi)\,d\mu(x,\xi)=0.
\end{eqnarray}
\end{Theorem}

Let us consider the wave operator in a inhomogeneous  medium:
\begin{eqnarray*}
	\rho(x) \partial_t^2 - \sum_{i,j=1}^n \partial_{x_i}\left[K(x) \partial_{x_j}\right].
\end{eqnarray*}

Using the notation $D_j=\frac{1}{i}\partial_j$ we can write
\begin{eqnarray*}
	P(t, x, D_t, D_{x})= - \rho(x) D_t^2 +   \sum_{i,j=1}^n D_{x_i}[K(x)D_{x_j}],\quad D_x=(D_{x_1},\cdots, D_{x_n}),
\end{eqnarray*}
whose principal symbol $p(t,x,\tau,\xi)$ is given by
\begin{eqnarray}\label{5.5}
p(t,x,\tau,\xi)=-\rho(x) \tau^2 + K(x) \,\xi\cdot \xi,\quad\xi=(\xi_1,\cdots,\xi_n),
\end{eqnarray}
where $t\in \mathbb{R}$,~ $x\in \Omega\subset \mathbb{R}^d$,~ $(\tau,\xi)\in \mathbb{R}\times \mathbb{R}^d$.

\begin{Proposition}\label{proposition5.1}
	Unless a change of variables, the bicharacteristics of \eqref{5.5} are curves of the form
	\begin{eqnarray*}
		t \mapsto \left(t, x(t), \tau, -\tau\left(\frac{K(x(t))}{\rho(x(t))}\right)^{-1}\dot{x}(t) \right),
	\end{eqnarray*}
	where $t \mapsto x(t)$ is a geodesic of the metric $G=\left(K/\rho\right)^{-1}$ on $\Omega$, parameterized by the curvilinear abscissa.
\end{Proposition}

The main result is the following:

 \begin{Theorem}\label{Theoremak 4.60}
 Let $P$ be a self-adjoint differential operator of order $m$ on $\mathcal{O}$ which admits a principal symbol $p$. Let $\{u_n\}_n$ be a bounded sequence in $L_{loc}^2(\mathcal{O})$ which converges weakly to zero, with a microlocal defect measure $\mu$. Let us assume that $P u_n$ converges to $0$ in $H_{loc}^{-(m-1)}(\mathcal{O})$. Then the support of $\mu$, $\hbox{supp}(\mu)$, is a union of curves like $s\in I \mapsto \left(x(s), \frac{\xi(s)}{|\xi(s)|}\right)$, where $s\in I \mapsto (x(s),\xi(s))$ is a bicharacteristic of $p$.
 \end{Theorem}

\end{document}